\documentclass[11pt]{amsart}

\usepackage{verbatim}

\usepackage{amsfonts}

\usepackage{amsthm}

\usepackage{amssymb}

\usepackage{amsmath}

\usepackage{amsfonts}

\usepackage{latexsym}

\usepackage{amscd}

\usepackage{enumerate}

\usepackage{hyperref}

\usepackage[all]{xy}

\usepackage{graphicx}

\usepackage{fancyhdr}

\long\def\symbolfootnote[#1]#2{\begingroup\def\thefootnote{\fnsymbol{footnote}}
\footnote[#1]{#2}\endgroup}

\newcommand{\N}{\mathbb{N}}

\newcommand{\Z}{\mathbb{Z}}

\newcommand{\R}{\mathbb{R}}

\newcommand{\C}{\mathbb{C}}

\newcommand{\h}{\mathcal{H}}

\newcommand{\Manoa}{M\=anoa}

\newcommand{\Hawaii}{Hawai\kern.05em`\kern.05em\relax i}

\newcommand{\dg}{\text{diag}}

\newcommand{\pex}{\partial_\mathcal{E}X}

\theoremstyle{plain} \newtheorem{main}{Theorem}[section]

\theoremstyle{definition} \newtheorem{cs}{Definition}[section]

\theoremstyle{definition} \newtheorem{cprops}[cs]{Definition}

\theoremstyle{definition} \newtheorem{spac}[cs]{Definition}

\theoremstyle{remark} \newtheorem{spacelem}[cs]{Examples}

\theoremstyle{remark} \newtheorem{wmrem}[cs]{Remark}

\theoremstyle{definition} \newtheorem{ops}[cs]{Definition}

\theoremstyle{plain} \newtheorem{dislem}[cs]{Lemma}

\theoremstyle{plain} \newtheorem{elde}[cs]{Lemma}

\theoremstyle{plain} \newtheorem{ptde}[cs]{Lemma}

\theoremstyle{plain}

\theoremstyle{definition} \newtheorem{uniroe}{Definition}[section]

\theoremstyle{definition} \newtheorem{repdef}[uniroe]{Definition}

\theoremstyle{plain} \newtheorem{conslem}[uniroe]{Lemma}

\theoremstyle{remark} \newtheorem{conex}[uniroe]{Example}

\theoremstyle{definition} \newtheorem{geotdef}[uniroe]{Definition}

\theoremstyle{remark} \newtheorem{wm}[uniroe]{Remark}

\theoremstyle{plain} \newtheorem{nclem}[uniroe]{Proposition}

\theoremstyle{remark} 

\theoremstyle{definition} 

\theoremstyle{plain} 

\theoremstyle{definition} \newtheorem{constant}[uniroe]{Definition}

\theoremstyle{definition} \newtheorem{cequiv}{Definition}[section]

\theoremstyle{remark} \newtheorem{incex}[cequiv]{Example}

\theoremstyle{plain} \newtheorem{cinv}[cequiv]{Theorem}

\theoremstyle{plain} \newtheorem{cequist}[cequiv]{Lemma}

\theoremstyle{plain} \newtheorem{bij}[cequiv]{Lemma}

\theoremstyle{plain} \newtheorem{inc}[cequiv]{Proposition}

\theoremstyle{plain} \newtheorem{aprops}[cequiv]{Lemma}

\theoremstyle{plain} \newtheorem{aprops2}[cequiv]{Lemma}

\theoremstyle{plain} \newtheorem{conn}[cequiv]{Lemma}

\theoremstyle{plain} \newtheorem{fin}[cequiv]{Lemma}

\theoremstyle{definition} \newtheorem{lapdef}{Definition}[section]

\theoremstyle{remark} \newtheorem{lapex}[lapdef]{Example}

\theoremstyle{plain} \newtheorem{poslem}[lapdef]{Lemma}

\theoremstyle{plain} \newtheorem{poscor}[lapdef]{Corollary}

\theoremstyle{definition} \newtheorem{spectra0}[lapdef]{Definition}

\theoremstyle{plain} \newtheorem{laplem}[lapdef]{Lemma}

\theoremstyle{plain} \newtheorem{conlem}[lapdef]{Proposition}

\theoremstyle{plain} \newtheorem{tlap}[lapdef]{Proposition}

\theoremstyle{definition} \newtheorem{expa}[lapdef]{Definition}

\theoremstyle{plain} \newtheorem{fin2}[lapdef]{Corollary}

\theoremstyle{remark} \newtheorem{exrem}[lapdef]{Remark}

\theoremstyle{plain} \newtheorem{amenlem}{Lemma}[section]

\theoremstyle{definition} \newtheorem{amendef}[amenlem]{Definition}

\theoremstyle{plain} \newtheorem{amen}[amenlem]{Proposition}

\theoremstyle{plain} \newtheorem{ameniso}[amenlem]{Lemma }

\theoremstyle{plain} \newtheorem{infi}[amenlem]{Corollary}

\theoremstyle{definition} \newtheorem{spectra}{Definition}[section]

\theoremstyle{definition} \newtheorem{bsdef}[spectra]{Definition}

\theoremstyle{plain} \newtheorem{tgeot}[spectra]{Theorem}

\theoremstyle{plain} \newtheorem{maxgp}[spectra]{Lemma}

\theoremstyle{remark} \newtheorem{injrem}[spectra]{Remark}

\theoremstyle{plain} \newtheorem{expgt}[spectra]{Corollary}

\theoremstyle{plain} \newtheorem{atthe}{Theorem}[section]

\theoremstyle{plain} \newtheorem{cslem}{Lemma}[section]

\theoremstyle{definition} \newtheorem{bhp}[cslem]{Definition}

\theoremstyle{plain} \newtheorem{girth}[cslem]{Theorem}

\theoremstyle{plain} 

\theoremstyle{plain} 

\theoremstyle{plain} \newtheorem{kposlem}[cslem]{Lemma}

\theoremstyle{plain} \newtheorem{funprops}[cslem]{Lemma}

\theoremstyle{plain} \newtheorem{graphlem}[cslem]{Lemma}

\theoremstyle{plain} \newtheorem{pitlem}[cslem]{Lemma}

\theoremstyle{plain} 

\theoremstyle{plain} 

\theoremstyle{plain} \newtheorem{invlem}[cslem]{Lemma}

\theoremstyle{plain} \newtheorem{speclem}[cslem]{Lemma}

\theoremstyle{remark} \newtheorem{questions}{Questions}[section]

\title{Geometric Property (T)}

\author{Rufus Willett}
\address{University of \Hawaii\ at \Manoa}
  \email{rufus@math.hawaii.edu}
\author{Guoliang Yu}
\address{Texas A\&M University, and Shanghai Center for Mathematical Sciences}
  \email{guoliangyu@math.tamu.edu}

\thanks{The authors were partly supported by the U.S. National Science Foundation.}

\begin{document}

\maketitle

\begin{abstract}
This paper discusses `geometric property (T)'.  This is a property of metric spaces introduced in earlier work of the authors for its applications to $K$-theory.  Geometric property (T) is a strong form of `expansion property': in particular for a sequence  $(X_n)$ of bounded degree finite graphs, it is strictly stronger than $(X_n)$ being an expander in the sense that the Cheeger constants $h(X_n)$ are bounded below.

In this paper, we show that geometric property (T) is a coarse invariant, i.e.\ depends only on the large-scale geometry of a metric space $X$. We also discuss how geometric property (T) interacts with amenability, property (T) for groups, and coarse geometric notions of a-T-menability.  In particular, we show that property (T) for a residually finite group is characterised by geometric property (T) for its finite quotients.
\end{abstract}


\section{Introduction}

In \cite[Section 7]{Willett:2010zh}, the current authors introduced \emph{geometric property (T)}: this is a pathological property of metric spaces designed to be an obstruction to the maximal version \cite{Gong:2008ja} of the coarse Baum-Connes conjecture in $K$-theory and higher index theory.  The treatment in \cite{Willett:2010zh} was quite brief; moreover, we expect that geometric property (T) will see some applications outside of $K$-theory and higher index theory.  It is the purpose of this paper to develop the theory more fully.

Throughout we will work with discrete metric spaces $X$ of \emph{bounded geometry}: this means that if $B(x;r)$ denotes the ball of radius $r$ about $x\in X$, then the quantity $\sup_{x\in X}|B(x;r)|$ is finite for all $r$.  We allow our metrics to take infinite distances.  For applications, the two most interesting examples of such spaces are: the vertex set of a graph equipped with the edge metric (for example, the Cayley graph of a finitely generated group); or the disjoint union of a sequence $(X_n)$ of finite graphs, where each $X_n$ has the edge metric, and the distance between different graphs is infinity (the coarse geometry of such a space is essentially the `asymptotic geometry' of the sequence).  The bounded geometry assumption amounts to the existence of an absolute bound on the degree of all vertices in either case.

Geometric property (T) for a discrete metric space $X$ says that unitary representations of the Gromov-Roe \emph{translation algebra} $\C_u[X]$ (see \cite[page 262]{Gromov:1993tr} and \cite[Chapter 4]{Roe:2003rw}) that have almost invariant vectors must have invariant vectors: see Definition \ref{geotdef} below.  This is a direct analogue of property (T) for a (discrete) group $\Gamma$, which says that any unitary representation of the group algebra $\C[\Gamma]$ having almost invariant vectors actually has invariant vectors.  In order to make sense of `invariant' and `almost invariant' in the metric space case one has to do a little work, but the basic idea is the same as in the group case.

Geometric property (T) can also be characterised (Proposition \ref{tlap} below) in terms of a spectral gap property for a \emph{Laplacian} operator $\Delta$ in $\C_u[X]$, much as was done by Valette for groups \cite[Theorem 3.2]{Valette:1984wy}.  In the case that $X$ is a graph (or built from a sequence of graphs), $\Delta$ simply is the graph Laplacian.  This was our original definition in \cite[Section 7]{Willett:2010zh}, but we found the almost invariant vectors version more convenient to work with in the current paper.\\

In this paper we establish the machinery needed to make rigorous sense of the above definitions.  We then prove the following results.
\begin{main}\label{main}
\begin{enumerate}
\item Geometric property (T) is a coarse invariant (Theorem \ref{cinv}).
\item Geometric property (T) for a sequence of finite graphs implies that the sequence is an expander (Corollary \ref{fin2}), but is strictly stronger than this (Corollary \ref{expgt}),
\item An infinite connected graph $X$ has geometric property (T) if and only if it is not amenable (Corollary \ref{infi}).
\item Let $\Gamma$ be a finitely generated discrete group, and $\Gamma=\Gamma_0\unrhd\Gamma_1\unrhd\cdots$ be a sequence of finite index normal subgroups such that $\cap_n\Gamma_n$ is the trivial subgroup.  Then $\Gamma$ has property (T) if and only if the sequence $(\Gamma/\Gamma_n)$ of finite Cayley graphs\footnote{Defined with respect to some fixed generating of $\Gamma$.} has geometric property (T) (Theorem \ref{tgeot}).
\item A sequence of finite graphs $(X_n)$ with geometric property (T) cannot also admit a fibered coarse embedding into Hilbert space \cite{Chen:2012uq}, or have the boundary Haagerup property \cite{Finn-Sell:2012fk}, unless $\sup_n|X_n|$ is finite (Theorem \ref{atthe}).
\end{enumerate}
\end{main}

A few remarks are in order.  Points (2) and (3) together suggest that geometric property (T) is not interesting for a single connected graph, but that it has serious content for a sequence of finite graphs.  Points (1) and (4) have the following consequence, which is perhaps surprising: for a residually finite group, property (T) can be characterised by the geometry of the finite quotients of the group; this contrasts with the well-known fact that property (T) for the group itself is not a geometric invariant \cite[Section 3.6]{Bekka:2000kx}.  Property (5) is a strong analogue in coarse geometry of the well-known incompatibility of property (T) and a-T-menability for groups.

\subsection*{Outline}
To facilitate algebraic computations involving $\C_u[X]$ we use the language of abstract coarse structures \cite{Roe:2003rw}, rather than the metric space language of the introduction, throughout the body of the paper; Section \ref{cssec} recalls the basic definitions of coarse structures and proves some combinatorial lemmas.  Section \ref{transsec} introduces the translation algebra $\C_u[X]$, the notion of invariant vectors in its representations, and geometric property (T).   Section \ref{cinvsec} proves that geometric property (T) is a coarse invariant; we could not find a short proof of this result and this is probably the most technical part of the paper.   Section \ref{lapsec} defines general combinatorial Laplacian operators, and characterises invariant vectors and geometric property (T) in terms of them.

Having established the basic properties of geometric property (T), the next three sections discuss the relationship between geometric property (T) and some other properties.  Section \ref{amensec} discusses the relationship of representations of $\C_u[X]$ with amenability, and uses this to characterise geometric property (T) for spaces as in part (3) of Theorem \ref{main}.  Section \ref{exsec} studies the relationship with property (T) for groups, proving part (4) of Theorem \ref{main}.  Section \ref{atmensec} discusses the relationship with the coarse a-T-menability properties in part (5) of Theorem \ref{main}.

We conclude the paper with some natural open questions in Section \ref{qsec}.

\subsection*{Acknowledgements}

We would like to thank Erik Guentner, J\'{a}n \v{S}pakula, and Romain Tessera for useful discussions on aspects of this work.   We would also like to thank Jintao Deng for pointing out some algebraic errors in an earlier version.  The first author would like to thank the Shanghai Center for Mathematical Sciences for its hospitality during part of the work on this paper.

\section{Coarse structures and some combinatorics}\label{cssec}

In this section we recall the definition of a coarse structure on a set $X$.  We then recall the definition of partial translation, and prove some combinatorial lemmas about decomposing general controlled sets into partial translations.

If $X$ is a set and $E,F$ are subsets of $X\times X$, then the \emph{composition} of $E$ and $F$, denoted $E\circ F$, is the set
$$
E\circ F:=\{(x,y)\in X\times X~|~\text{ there exists $z$ such that } (x,z)\in E \text{ and } (z,y)\in F\}
$$
and the \emph{inverse} of $E$ is
$$
E^{-1}:=\{(x,y)\in X\times X~|~(y,x)\in E\}.
$$
For $n\geq 1$, we use the shorthand
$$
E^{\circ n}:=\underbrace{E\circ \cdots \circ E}_{n}.
$$
Finally, we will write $\dg(E)$ for the `diagonal part' of $E$, that is
$$
\dg(E):=E\cap \{(x,x)\in X\times X~|~x\in X\}.
$$

\begin{cs}\label{cs}
Let $X$ be a set.  A \emph{coarse structure} on $X$ consists of a collection $\mathcal{E}$ of subsets of $X\times X$ such that:
\begin{enumerate}
\item the diagonal $\{(x,x)\in X\times X~|~x\in X\}$ is in $\mathcal{E}$;\footnote{This condition is not always assumed: coarse structures satisfying this condition are sometimes called \emph{unital}.}
\item if $E\in \mathcal{E}$ and $F\subseteq E$, then $F\in \mathcal{E}$;
\item if $E,F\in \mathcal{E}$ then $E\circ F\in \mathcal{E}$;
\item if $E\in \mathcal{E}$, then $E^{-1}\in\mathcal{E}$.
\end{enumerate}
The members of $\mathcal{E}$ are called \emph{controlled sets} for the coarse structure.  

A set $X$ equipped with a coarse structure is called a \emph{coarse space}.
\end{cs}  

The motivating example comes when $X$ is a metric space, and a set is controlled if and only if it is a subset of a `tube' $\{(x,y)\in X\times X~|~d(x,y)\leq r\}$ for some $r>0$.  

The following definition lists some additional properties of controlled sets and coarse spaces that we will need. 

\begin{cprops}\label{cprops}
Let $X$ be a coarse space, and $\mathcal{E}$ the coarse structure on $X$.  
\begin{itemize}
\item A controlled set $E$ is called \emph{symmetric} if $E=E^{-1}$.
\item The coarse structure is said to have \emph{bounded geometry} if for any controlled set $E$ there is a bound $M=M(E)$ such that for any $x\in X$,
$$
|\{y\in X~|~(x,y)\in E\cup E^{-1}\}|\leq M.
$$
\item A controlled set $E$ is said to be \emph{generating} if for any controlled set $F$ there exists $n$ such that $F\subseteq E^{\circ n}$.  A coarse space is said to be \emph{monogenic} if there exists a generating controlled set for the coarse structure.
\item Two elements $x,y$ of $X$ are \emph{in the same coarse component} of $X$ if the set $\{(x,y)\}$ is controlled.  `Being in the same coarse component' defines an equivalence relation on $X$, and the equivalence classes are called \emph{coarse components}.  If there is only a single coarse component, $X$ is said to be \emph{coarsely connected}.  
\item If $Y$ is a subset of $X$ it is itself a coarse space with the controlled sets being the intersection of the controlled sets for $X$ with $Y\times Y$.  This is called the \emph{induced coarse structure}, and will be used implicitly many times below.
\end{itemize}
\end{cprops}

\begin{spac}\label{spac}
We will say \emph{$X$ is a space} as an abbreviation for `$X$ is a bounded geometry, monogenic coarse space, with at most countably many coarse components'.  
\end{spac}

Note that spaces are automatically countable; this and being monogenic implies that the coarse structure on a space always comes from a metric \cite[Section 2.5]{Roe:2003rw}, with possibly infinite distances.  Nonetheless, the language of abstract coarse structures is more convenient for the computations in this paper.

The reader will probably find it useful to keep the following examples in mind.

\begin{spacelem}\label{spacelem}
Let $X$ be the vertex set of an undirected graph, and $E$ the set of edges, which we consider as a symmetric subset of $X\times X$.  The coarse structure generated by the set $E$ is monogenic, and is bounded geometry if and only if there is a uniform  bound on the degrees of all vertices in $X$.  The coarse components of $X$ are exactly the (vertex sets of the) connected components of the underlying graph.  The coarse structure above is the same as that defined by the \emph{edge metric}, which sets the distance between two vertices to be the shortest number of edges in a path between them, and infinity if no such path exists.  

Particularly important classes of examples are Cayley graphs of discrete groups, and discretisations of Riemannian manifolds.   Another import example for us occurs when $X$ is a disjoint union $X=\sqcup X_n$ of finite connected graphs: examples of this form are important in coarse geometry as they are relatively easy to analyse, and as questions about general spaces can often be reduced to questions about spaces of this form.
\end{spacelem}

\begin{wmrem}\label{wmrem}
Let $X=\sqcup X_n$ be a disjoint union of finite connected graphs as in Example \ref{spacelem} above.  It is common in coarse geometry (in order to avoid infinite-valued metrics) to metrise such spaces with any metric that restricts to the edge metric on the individual $X_n$ and satisfies $d(X_n,X\setminus X_n)\to\infty$ as $n\to\infty$.  The corresponding coarse structure is not monogenic, but is `weakly monogenic' in the following sense.   A controlled set $E$ is said to be a \emph{weak generating set} for the coarse structure if for any controlled set $F$ there exists $n\in \N$ such that $F\setminus E^{\circ n}$ is finite.  A coarse space is said to be \emph{weakly monogenic} if there is a weak generating set for the coarse structure.  Most of the results of this paper hold for weakly monogenic spaces, up to minor adjustments (see Remark \ref{wm} below), so can be applied to such spaces directly.  
\end{wmrem}

The following definition is based on \cite[Definition 8]{Brodzki:2007mi}.  

\begin{ops}\label{ops}
Let $X$ be a space.  A \emph{partial translation} on $X$ consists of the following data: subsets $A$ and $B$ of $X$ and a bijection $t:A\to B$ such that the \emph{graph}\footnote{We have defined the graph of $t$ the `wrong way round' to better match matrix multiplication later.} of $t$
$$
\text{graph}(t):=\{(t(x),x)\in X\times X~|~x\in A\}
$$
is controlled.  The subset $A$ is called the \emph{support} of $t$, $B$ its \emph{range}.   The \emph{inverse} of a partial translation $t:A\to B$ is the partial translation $t^{-1}:B\to A$ defined by inverting the bijection $t$.

A controlled set $E$ is called \emph{elementary} if there exists a (necessarily unique) partial translation $t:A\to B$ such that $E=\text{graph}(t)$.  An elementary controlled set is called \emph{antisymmetric} if the domain and range of the corresponding partial translation do not intersect\footnote{Equivalently, the images of the two coordinate projections are disjoint when restricted to $E$.}.
\end{ops}

In the remainder of this section we prove some combinatorial lemmas about decomposing controlled sets into partial translations; these will be useful for algebraic computations later in the paper.

The following very general lemma is probably well-known.

\begin{dislem}\label{dislem}
Let $A$ be a set, and $B$ and $C$ be subsets of $A$.  Let $t:B\to C$ be a bijection such that $t(a)\neq a$ for all $a\in B$.  Then there exists a decomposition
$$
B=B_0\sqcup B_1\sqcup B_2
$$
of $B$ into (at most) three disjoint subsets such that $t(B_i)\cap B_i=\varnothing$ for all $i\in \{0,1,2\}$.
\end{dislem}

The example $A=B=C=\{1,2,3\}$, $t$ a cyclic permutation, shows that one cannot get away with less than three subsets.

\begin{proof}
Let $s:C\to B$ be any bijection which is the identity on $C\cap B$.  Replacing $t$ with $s\circ t$, it is not difficult to see that it suffices to prove the following statement: if $B$ is a set and $t:B\to B$ is a bijection such that $t(b)\neq b$ for all $b\in B$, then there exists a decomposition $B=B_0\sqcup B_1\sqcup B_2$ such that $t(B_i)\cap B_i=\varnothing$ for all $i\in \{0,1,2\}$.  We now prove this.

The bijection $t:B\to B$ gives rise to an action of $\Z$, which partitions $B$ into orbits.  As $t(b)\neq b$ for all $b\in B$, each orbit for this action has one of the following forms.  
\begin{enumerate}[(1)]
\item $\{...,t^{-2}(b),t^{-1}(b),b=t^0(b),t(b),t^2(b),...\}$ (going on infinitely in both directions) for some $b\in B$.
\item $\{b=t^0(b),t(b),...,t^n(b)\}$ for some $n\geq 1$ and $b\in B$ such that $t^{n+1}(b)=b$.
\end{enumerate}
Define subsets $B_0$, $B_1$ and $B_2$ of $B$ as follows.  For each orbit, fix once and for all a representation of one of the types above.  For an orbit of type (2) with $n$ even and $i=n$, put $t^i(b)$ into $B_2$.  In all other cases, put $t^i(b)$ into $B_{i\text{ mod } 2}$ (where $i \text{ mod }2$ is always construed as $0$ or $1$).  A routine case-by-case analysis shows that this works.
\end{proof}

\begin{elde}\label{elde}
Let $E\subseteq F$ be symmetric controlled sets on a space $X$.  Then there exist elementary controlled sets $E_1,...,E_n$ such that $F$ is the disjoint union
$$
F=E\sqcup \dg(F\backslash E)\sqcup \bigsqcup_{i=1}^n(E_i \sqcup E_i^{-1}).
$$
We may assume moreover that each $E_i$ is antisymmetric. 
\end{elde}

\begin{proof}
Inductively define $E_0=E$ and  $E_{i+1}$ to be any maximal elementary subset of 
$$
F\backslash(E\cup \dg(F\backslash E)\cup (E_1\cup E_1^{-1})\cup\cdots\cup (E_i\cup E_i^{-1}))
$$
such that $E_{i+1}\cap E_{i+1}^{-1}=\varnothing$.  

Assume that $E_{i+1}$ is not empty, and assume that $(x,y)$ is in $E_{i+1}$, noting that this forces $x\neq y$.  Then  maximality of each $E_j$ forces the existence of distinct $y_0,...,y_i$ such that $(x,y_j)$ is in $E_j\cup E_j^{-1}$ for each $j=0,...,i$.  In particular, 
$$
|\{z\in X~|~(x,z)\in F\}|\geq i+2,
$$
which is impossible for $i$ suitably large by the bounded geometry condition.  Thus $F=E\cup \dg(F\backslash E)\cup (E_1\cup E_1^{-1})\cdots \cup (E_n \cup E_n^{-1})$ for some $n$.  

Finally, note that Lemma \ref{dislem} applied to the partial translation underlying each $E_i$ decomposes $E_i$ into three antisymmetric parts.  Decomposing further, we may thus assume that each $E_i$ is antisymmetric.
\end{proof}

\begin{ptde}\label{ptde}
Let $t:A\to B$ be a partial translation on a space $X$ and $E$ be a controlled set for $X$.  Assume that
$$
E^{\circ n}\supseteq \text{graph}(t)
$$
for some $n\geq 1$.

Then there exists a decomposition $A=A_1\sqcup \cdots \sqcup A_m$ such that if $t_i$ is the restriction of $t$ to $A_i$ then there exist partial translations $\{s_i^j\}_{i=1,j=1}^{m,~~~n}$ such that
\begin{itemize}
\item $t_i=s_i^1\circ \cdots \circ s_i^n$;
\item $\text{graph}(s_i^j)\subseteq E$ for all $i=1,...,m$ and $j=1,...,n$;
\item for each $i$ and each $j=1,...,n-1$, the range of $s_i^j$ is equal to the domain of $s_i^{j+1}$.
\item for each $i$ and $j=1,...,n$, either $s^j_i$ is the identity map, or the range of $s^j_i$ is disjoint from its support.
\end{itemize}
\end{ptde}

\begin{proof}
For each $x\in A$ the pair $(t(x),x)$ is contained in $E^{\circ n}$, whence we may choose points $x=r_0(x),r_1(x),\cdots r_n(x)=t(x)$ such that for each $j=1,...,n$, the pair $(r_{j}(x),r_{j-1}(x))$ is in $E$.   In this way, define functions $r_j:A\to X$.  

Note that the bounded geometry assumption and the fact that the graph of $r_1$ is contained in $E$ implies that there exists $N_1$ such that $A$ decomposes into $N_1$ sets $A^1_1,...,A^1_{N_1}$ such that the following hold:
\begin{itemize}
\item $r_1$ is a bijection restricted to each $A_i^1$;
\item either $r_1(x)=r_0(x)$ for all $x\in A_i^1$, or $r_1(A_i^1)\cap r_0(A_i^1)=\varnothing$.
\end{itemize}
Similarly, as the graph of $r_2$ is contained in $E^{\circ 2}$ there exists $N_2$ such that each $A^1_i$ decomposes into at most $N_2$ sets $A^2_{ij}$, for which the restriction of $r_2$ to each $A^2_{ij}$ is a bijection, and such that either $r_2(x)=r_1(x)$ for all $x\in A_{ij}^2$, or $r_2(A^2_{ij})\cap r_1(A^2_{ij})=\varnothing$.   Continuing in this way, we get a decomposition $A=A_1,...,A_m$, where $m$ is at most $N_1N_2\cdots N_n$, such that:
\begin{enumerate}
\item each $r_j$ is a bijection when restricted to each $A_i$;
\item for each $i,j$ either $r_j(x)=r_{j-1}(x)$ for all $x\in A_i$, or $r_j(A_i)\cap r_{j-1}(A_i)=\varnothing$.
\end{enumerate}

Define for each $i=1,...,m$ a function $s_i^j:r_{j-1}(A_i)\to r_{j}(A_i)$ by the stipulation 
$$
s_i^j(r_{j-1}(x))=r_j(x);
$$ 
as each $r_j$ is injective on each $A_i$, this is well-defined and bijective.  It is not difficult to see that these functions $s_i^j$ have the right properties.
\end{proof}



\section{Translation algebras and geometric property (T)}\label{transsec}

In this section we introduce translation algebras and define geometric property (T) in terms of their (`unitary') representation theory.

 Throughout this section, $X$ denotes a space in the sense of Definition \ref{spac}.

\begin{uniroe}\label{uniroe}
The \emph{translation algebra}, or \emph{algebraic uniform Roe algebra} of $X$, denoted $\C_u[X]$, is the collection of all $X$-by-$X$ indexed matrices $T=(T_{xy})_{x,y\in X}$ with entires in $\C$ such that
$$
\sup_{x,y\in X}|T_{xy}|
$$
is finite, and such that for any $T\in \C_u[X]$ the \emph{support} of $T$ defined by
$$
\text{supp}(T):=\{(x,y)\in X\times X~|~T_{xy}\neq 0\}
$$
is a controlled set.  The usual matrix operations and adjoint make $\C_u[X]$ into a $*$-algebra.
\end{uniroe}
Note that the collection of matrices in $\C_u[X]$ supported on the diagonal constitutes a copy of $l^\infty(X)$ inside $\C_u[X]$.  

Partial translations give rise to operators in $\C_u[X]$ in the following way.  Let $t:A\to B$ be a partial translation.  Then $t$ gives rise to an operator $v\in \C_u[X]$ defined by setting
$$
v_{xy}=\left\{\begin{array}{ll} 1 & t(y)=x \\ 0 & \text{otherwise} \end{array}\right..
$$
An operator arising in this way is called a \emph{partial translation}; note that $t$ and $v$ determine each other uniquely, so there should not be any confusion caused by the repeated terminology.  It is immediate form the definitions that if $v$ is a partial translation operator corresponding to $t:A\to B$ then $v$ is a partial isometry, with $v^*$ the partial translation operator corresponding to $t^{-1}$.  Moreover, the support and range projections $v^*v$ and $vv^*$ are the characteristic functions of $A$, $B$ respectively, considered as elements of the diagonal $*$-subalgebra $l^\infty(X)$, and the support of $v$ is the graph of $t$.

\begin{repdef}\label{repdef}
A \emph{representation} of $\C_u[X]$ is by definition a unital $*$-homomorphism $\pi:\C_u[X]\to\mathcal{B}(\mathcal{H})$ from $\C_u[X]$ to the $C^*$-algebra of bounded operators on some Hilbert space $\h$.  We will usually leave $\pi$ implicit, saying just that $\h$ is a representation of $\C_u[X]$, and writing $T\xi$ for the image of an element $\xi$ of $\h$ under $\pi(T)$.
\end{repdef}
The assumption that representations are unital in the above is not very important, but does not significantly reduce generality and streamlines some arguments slightly.  In contrast, the assumption that all representations are $*$-preserving is crucial; such representations should be thought of as the analogues of unitary representations of a group.

\begin{constant}\label{constant}
Let $\h$ be a representation of $\C_u[X]$.  A vector $\xi\in\h$ is said to be \emph{invariant}, or \emph{constant}, if $v\xi=vv^*\xi$ for all partial translations $v$,.

The constant elements form a closed subspace\footnote{It is not a subrepresentation in general.} of $\h$, which we denote $\h_c$.
\end{constant}

\begin{conex}\label{conex}
If $X=\sqcup X_n$ is a disjoint union of finite connected graphs as in Example \ref{spacelem}, then the constant vectors in $l^2(X)$ are exactly those square-summable functions on $X$ that are constant on each coarse component $X_n$. 
\end{conex}

Geometric property (T) says that for any representation $\h$ of $\C_u[X]$, vectors in $\h_c^\perp$ cannot be `too close' to constant.  Here is the formal definition.

\begin{geotdef}\label{geotdef}
A space $X$ has \emph{geometric property (T)} if for any controlled generating set $E$ there exists a constant $c=c(E)>0$ such that for any representation $\h$ and $\xi\in \h_c^\perp$ there exists a partial translation $v$ in $\C_u[X]$ with support in $E$ such that 
$$
\|(vv^*-v)\xi\|\geq c\|\xi\|.
$$ 
\end{geotdef}

The reader should compare this to the following definition of property (T) for a discrete group (compare \cite[Section 1.1]{Bekka:2000kx}).  For a unitary representation of a finitely generated group $\Gamma$, let $\h_c$ denote the constant vectors: those $\xi\in \h$ for which $g\xi=\xi$ for all $g\in \Gamma$.  A finitely generated\footnote{Property (T) forces finite generation on a discrete group, so there is no harm assuming this.} group $\Gamma$ then has \emph{property (T)} if for any finite generating set $E$ of $\Gamma$ there exists a constant $c=c(E)>0$ such that for any unitary representation $\h$ of $\Gamma$ and any $\xi\in \h_c^\perp$ there exists $g\in E$ with 
$$
\|(gg^*-g)\xi\|\geq c\|\xi\|.
$$

\begin{wm}\label{wm}
A representation of $\C_u[X]$ is called a \emph{boundary representation} if it contains the ideal 
$$
\C_f[X]:=\{T\in \C_u[X]~|~T_{xy}\neq 0 \text{ for only finitely many }x,y\}
$$ 
in its kernel.  Geometric property (T) can be weakened to \emph{boundary property (T)} by requiring that the property in Definition \ref{geotdef} above holds only for all boundary representations.  This notion is more appropriate for weakly monogenic coarse spaces as discussed in Remark \ref{wmrem}.  Indeed, the results in this paper all continue to hold for weakly monogenic bounded geometry coarse spaces (with obvious minor variations) if `generating' is replaced by `weakly generating', `representation' by `boundary representation' and `geometric property (T)' by `boundary property (T)' everywhere.
\end{wm}

In the remainder of this section, we give some equivalent formulations of geometric property (T) that will be useful later.

Define a linear map $\Phi:\C_u[X]\to\l^\infty(X)$ by 
\begin{equation}\label{phidef}
\Phi(T):x\mapsto \sum_{y\in X}T_{xy}.
\end{equation}
The map $\Phi$ can be used to characterise constant vectors as follows.

\begin{conslem}\label{conslem}
Let $\xi$ be a vector in a representation $\h$ of $\C_u[X]$.  Then the following are equivalent.
\begin{enumerate}
\item For all $T\in\C_u[X]$, $T\xi=\Phi(T)\xi$.
\item For all partial translations $v$ in $\C_u[X]$, $vv^*\xi=v\xi$.
\end{enumerate}
\end{conslem}

\begin{proof}
For a partial translation $v$, $\Phi(v)=vv^*$, so clearly (1) implies (2).  Assume that $\xi$ satisfies (2), and let $T$ be an element of $\C_u[X]$.  We may write $T$ as a finite sum
$$
T=\sum_{i=1}^n f_i v_i,
$$
where each $v_i$ is a partial translation, and each $f_i$ is the element of $l^\infty(X)$ defined by 
$$
f_i(x)=\left\{\begin{array}{ll} T_{xy} & (v_i)_{xy}=1 \\ 0 & \text{ otherwise}\end{array}\right.
$$
Noting that $f_iv_iv_i^*=f_i$ for all $i$, we then have that
$$
T\xi=\sum_{i=1}^n f_iv_i\xi=\sum_{i=1}^n f_iv_iv_i^*\xi=\sum_{i=1}^nf_i \xi =\Phi(T)\xi. 
$$
as required.
\end{proof}

Here then is the promised equivalent formulation of geometric property (T).

\begin{nclem}\label{nclem}
The following are equivalent.
\begin{enumerate}
\item $X$ has geometric property (T).
\item There exists a controlled generating set $E$ and a constant $c>0$ such that for any representation $\h$ and $\xi\in \h_c^\perp$ there exists a partial translation $v$ in $\C_u[X]$ with support in $E$ such that 
$$
\|(vv^*-v)\xi\|\geq c\|\xi\|.
$$
\item For any controlled generating set $E$ there exists a constant $c=c(E)>0$ such that for any representation $\h$ and $\xi\in \h_c^\perp$ there exists an operator $T\in \C_u[X]$ with support in $E$ such that
$$
\|(T-\Phi(T))\xi\|>c\sup_{x,y}|T_{xy}|\|\xi\|.
$$
\item  There exists a controlled generating set $E$ and a constant $c>0$ such that for any representation $\h$ and $\xi\in \h_c^\perp$ there exists an operator $T\in \C_u[X]$ with support in $E$ such that
$$
\|(T-\Phi(T))\xi\|>c\sup_{x,y}|T_{xy}|\|\xi\|.
$$
\end{enumerate}
\end{nclem}

\begin{proof}
The implication (1) implies (2) is clear.  For the converse, assume that $E$ and $c>0$ are as in (2), and let $F$ be any controlled generating set for the coarse structure.  As $F$ is generating, there exists $n$ such that $F^{\circ n}$ contains $E$, and thus property (2) holds with $F^{\circ n}$ replacing $E$.  Now, let $\h$ be a representation of $\C_u[X]$, and $\xi$ a unit vector in $\h_c^\perp$.  Using property (2) for $F^{\circ n}$, there exists a partial translation $t:A\to B$ with graph contained in $F^{\circ n}$ such that if $v$ is the corresponding operator, then $\|vv^*\xi-v\xi\|>c$.  

Now, using Lemma \ref{ptde} there exist partial translations $v_1,...,v_n$ such that $v=v_1\cdots v_n$, so that $\text{supp}(v_i)\subseteq F$ for each $i$, and so that $v_iv_i^*=v_{i-1}^*v_{i-1}$ for all $i=2,...,n$.  We then have that
\begin{align*}
c&\leq\|(vv^*-v)\xi\| \\ &= \|(v_1\cdots v_nv_n^*\cdots v_1^*-v_1\cdots v_n)\xi\| \\ &= \|v_1\cdots v_{n-1}v_{n-1}^*\cdots v_1^*-v_1\cdots v_n)\xi\| \\
&\leq \|(v_1\cdots v_{n-1}v_{n-1}^*\cdots v_1^*-v_1\cdots v_{n-1})\xi\|+\|(v_1\cdots v_{n-1}-v_1\cdots v_n)\xi\| \\
&\leq \|(v_1\cdots v_{n-1}v_{n-1}^*\cdots v_1^*-v_1\cdots v_{n-1})\xi\|+\|(v_{n-1}-v_{n-1}v_n)\xi\| \\
&= \|(v_1\cdots v_{n-1}v_{n-1}^*\cdots v_1^*-v_1\cdots v_{n-1})\xi\|+\|(v_{n-1}v_nv_n^*-v_{n-1}v_n)\xi\| \\
&\leq \|(v_1\cdots v_{n-1}v_{n-1}^*\cdots v_1^*-v_1\cdots v_{n-1})\xi\|+\|(v_nv_n^*-v_n)\xi\|. 
\end{align*}
Continuing in this way, we may conclude that
$$
c\leq\sum_{i=1}^n \|(v_iv_i^*-v_i)\xi\|,
$$
whence for some $i=1,...,n$, $\|(v_iv_i^*-v_i)\xi\|\geq c/n$.  We may thus take $c(F)=c/n$.

We will now show that (2) and (4) are equivalent; the proof that (1) and (3) are equivalent is analogous.  Noting as in the proof of Lemma \ref{conslem} that for a partial translation $v$, $\Phi(v)=vv^*$ it is clear that (2) implies (4), so it suffices to show that (4) implies (2).  

Let then $E$ and $c>0$ be as in the statement of (4).  Let $\h$ be a representation of $\C_u[X]$, let $\xi$ be an element of $\h_c^\perp$, and let $T\in \C_u[X]$ be as in the statement of (4) for this $\xi$.  We may write $T=\sum_{i=1}^n f_i v_i$ as in the proof of Lemma \ref{conslem}, where $n$ depends only on $E$ (not on $T$, $\h$ or $\xi$) and each $f_i$ has norm at most $\sup_{x,y\in X}|T_{xy}|$ as an element of $l^\infty(X)$.  As $*$-representations of the $C^*$-algebra $l^\infty(X)$ are contractive, each $f_i$ also has norm at most $\sup_{x,y\in X}|T_{xy}|$ when considered as an operator on $\mathcal{H}$.  We have then that
\begin{align*}
c\sup_{x,y\in X}|T_{xy}|\|\xi\|& < \|(T-\Phi(T))\xi\|\leq \sum_{i=1}^n \|(f_iv_i-f_i)\xi\|= \sum_{i=1}^n \|(f_iv_i-f_iv_iv_i^*)\xi\| \\&\leq \sum_{i=1}^n \|f_i\|\|(v_i-v_iv_i^*)\xi\|\leq \sup_{x,y\in X}|T_{xy}|\sum_{i=1}^n \|(v_i-v_iv_i^*)\xi\|.
\end{align*}
Hence for some $i$, $\|(v_i-v_iv_i^*)\xi\|\geq (c/n)\|\xi\|$; as $n$ depends only on $E$, this implies (2).
\end{proof}

\section{Coarse invariance}\label{cinvsec}

In this section we show that geometric property (T) is a coarse invariant, i.e.\ invariant under coarse equivalences as in the following definition.

\begin{cequiv}\label{cequiv}
A function $f:X\to Y$ between two spaces is \emph{uniformly expansive} if for any controlled set $E$ for $X$, the set
$$
\{(f(x_1),f(x_2))\in Y\times Y~|~(x_1,x_2)\in E\}
$$
is controlled for $Y$.  Two functions $f,g:X\to Y$ between two spaces are \emph{close} if the set
$$
\{(f(x),g(x))~|~x\in X\}
$$
is controlled for $Y$.

Two spaces $X$ and $Y$ are \emph{coarsely equivalent} if there exist uniformly expansive functions 
$$
f:X\to Y,~~~g:Y\to X
$$
such that the compositions $f\circ g$ and $g\circ f$ are close to the identities on $Y$ and $X$ respectively.
\end{cequiv}

The following example will be important in the proofs that follow.

\begin{incex}\label{incex}
A subset $Y$ of a space $X$ (with the inherited coarse structure) is \emph{coarsely dense} if there is a controlled set $E$ for $X$ such that the set
\begin{equation}\label{nearby}
\{y\in Y~|~(x,y)\in E\}
\end{equation}
is non-empty for all $x\in X$.  It is not difficult to see that $Y$ is coarsely dense if and only if the inclusion $i:Y\to X$ is a coarse equivalence, with `the inverse-up-to-closeness' given by any function $p:X\to Y$ that takes each $x\in X$ to any $y$ in the set in line \eqref{nearby} above.
\end{incex}

Our main goal in this section then is to prove the following result: we stated we expected this to be true in \cite[Section 7]{Willett:2010zh}, but did not have a complete proof at that time.

\begin{cinv}\label{cinv}
Let $X$ and $Y$ be coarsely equivalent spaces.  Then $X$ has geometric property (T) if and only if $Y$ does.
\end{cinv}

We start with a well-known `structural result' about coarse equivalences.

\begin{cequist}\label{cequist}
Let $f:X\to Y$ be a coarse equivalence.  Then there exist coarsely dense subspaces $X'$ of $X$ and $Y'$ of $Y$ such that $f$ restricts to a bijection $f':X'\to Y'$.  In other words, for any coarse equivalence $f:X\to Y$, there is a factorization
$$
\xymatrix{ X \ar[d]^p \ar[r]^f & Y \\ X' \ar[r]^{g} & Y' \ar[u]^i},
$$
where $p:X\to X'$ is an inverse-up-to-closeness of the inclusion of $X'$ in $X$, $g$ is a bijective coarse equivalence, and $i:Y'\to Y$ is the inclusion of a coarsely dense subset.
\end{cequist}

\begin{proof}
Let $Y'=f(X)$.  For each $y\in Y'$, choose $x(y)\in f^{-1}(y)$, and define $X'=\{x(y)\in X~|~y\in Y\}$.  It is not difficult to check that this $X'$ and $Y'$ have the required properties.
\end{proof}

To prove Theorem \ref{cinv}, it will thus suffice to prove the following two results.

\begin{bij}\label{bij}
Let $f:X\to Y$ be a bijective coarse equivalence.  Then $X$ has geometric property (T) if and only if $Y$ does.
\end{bij}

\begin{inc}\label{inc}
Let $Y$ be a coarsely dense subspace of a space $X$.  Then $Y$ has geometric property (T) if and only if $X$ does.
\end{inc}

\begin{proof}[Proof of Lemma \ref{bij}]
Define a function $f^*:\C_u[Y]\to \C_u[X]$ by $f^*(T)_{x_1x_2}:=T_{f(x_1)f(x_2)}$.  It is not difficult that $f^*$ is a $*$-isomorphism that restricts to a bijection between the collections of partial translations in $\C_u[Y]$ and $\C_u[X]$.  The result follows immediately from this. 
\end{proof}

The proof of Proposition \ref{inc} is more involved.  For the benefit of those readers who know about Morita equivalence, we explain the basic idea as follows.  We will define a projection $A\in \C_u[X]$ such that 
$$
A\C_u[X]A\cong \C_u[Y],~~~\C_u[X]A\C_u[X]=\C_u[X],
$$
so $A$ is a full projection, implementing a `$*$-algebra Morita equivalence' between $\C_u[X]$ and $\C_u[Y]$.  This Morita equivalence implements a bijective correspondence between the sets of representations of $\C_u[X]$ and $\C_u[Y]$ roughly defined by
\begin{equation}\label{corr}
\begin{array}{rllrll}
\text{Rep}(\C_u[X])& \to& \text{Rep}(\C_u[Y]), &\text{Rep}(\C_u[Y]) & \to& \text{Rep}(\C_u[X]) \\
\h & \mapsto& A\cdot \h &\h &\mapsto & \C_u[X]A\otimes_{\C_u[Y]}\h.
\end{array}
\end{equation}
The projection $A=\chi_Y$, the characteristic function of $Y$ in $l^\infty(X)$, has the above properties, but it does not behave well with respect to constant vectors.  We will thus take $A$ to be a sort of `averaging operator': this has the crucial property that the correspondences in line \eqref{corr} above `almost' takes constant vectors to constant vectors.

Now for the details.  We require some notational preliminaries.  Fix a decomposition $X=\sqcup_{y\in Y} U_y$ of $X$ into subsets $U_y$ parametrized by $Y$ such that for each $y\in Y$, $U_y$ contains $y$, and so there is a controlled set $E$ such that $U_y\times U_y\subseteq E$ for all $y$ (and in particular, $\max_{y\in Y}|U_y|$ is finite); it is not difficult to see that coarse denseness of $Y$ in $X$ implies that such a decomposition exists.  For $x\in X$, write $y(x)$ for the (unique) $y\in Y$ such that $x$ is in $U_y$.  For $y\in Y$, define $n(y)=|U_y|$.  For $x$ in $X$ we define also
\begin{equation}\label{bign}
n(x):=|U_{y(x)}|=n(y(x)),~~~N(x):=n(x)^{\frac{1}{2}}.
\end{equation}
We will think of $N$ as an (invertible) element of $l^\infty(X)\subseteq \C_u[X]$.

Define now an operator $A$ in $\C_u[X]$ by the formula
$$
A_{xz}=\left\{\begin{array}{ll} n(x)^{-1}  & y(z)=y(x)\\0 & \text{ otherwise}\end{array}\right..
$$
The operator $A$ can be thought of as an `averaging operator': as an operator on $l^2(X)$ it is the orthogonal projection onto the subspace of functions that are constant on each $U_y$.  Note that $A$ commutes with $N$.

The proof of Proposition \ref{inc} now proceeds via a series of (mainly algebraic) lemmas.

\begin{aprops}\label{aprops}
The following hold for the operator $A$.
\begin{enumerate}
\item If $\Phi:\C_u[X]\to l^\infty(X)$ is as in line \eqref{phidef} above, then $\Phi(A)$ is the constant function $1$.
\item If $\h$ is any representation of $\C_u[X]$, then the constant vectors are a subspace of $A\cdot \h$.
\item An element $T$ of $\C_u[X]$ is in $A\C_u[X]A$ if and only if $T_{x_1z_1}=T_{x_2z_2}$ whenever $y(x_1)=y(x_2)$ and $y(z_1)=y(z_2)$.
\item The maps $\alpha:A\C_u[X]A\to \C_u[Y]$ and $\beta:\C_u[Y]\to A\C_u[X]A$ defined by 
\begin{equation}\label{alpha}
\alpha(T)_{y_1y_2}=n(y_1)^{-\frac{1}{2}}n(y_2)^{-\frac{1}{2}}\sum_{x\in U_{y_1},z\in U_{y_2}}T_{xz},
\end{equation}
and 
\begin{equation}\label{beta}
\beta(T)_{xz}=n(x)^{-\frac{1}{2}}n(z)^{-\frac{1}{2}}T_{y(x)y(z)}
\end{equation}
are mutually inverse $*$-isomorphisms between $A\C_u[X]A$ and $\C_u[Y]$.
\item The set $\{TAS\in \C_u[X]~|~T,S\in \C_u[X]\}$ spans $\C_u[X]$.
\end{enumerate}
\end{aprops}

\begin{proof}
For (1), we clearly have
$$
\Phi(A)(x)=\sum_{z\in U_{y(x)}}n(x)^{-1}=1.
$$
Part (2) follows from part (1) and Lemma \ref{conslem}.

For (3), let $T$ be an element of $\C_u[X]$ and $x,z\in X$.  Then
$$
(TA)_{xz}=\sum_{u\in X} T_{xu}A_{uz}=n(z)^{-1}\sum_{u\in U_{y(z)}}T_{xu},~~~
(TA)_{xz}=n(x)^{-1}\sum_{u\in U_{y(x)}}T_{uz}.
$$
The claim follows from these formulas.  We will often implicitly use these formulas from now on.

For (4), note that $\alpha$ is clearly linear and $*$-preserving.  Note also that for $S,T\in A\C_u[X]A$,
\begin{align*}
\alpha(TS)_{y_1y_2} & =n(y_1)^{-\frac{1}{2}}n(y_2)^{-\frac{1}{2}}\sum_{x\in U_{y_1},z\in U_{y_2}}\sum_{u\in X}T_{xu}S_{uz} \\ & 
=n(y_1)^{-\frac{1}{2}}n(y_2)^{-\frac{1}{2}}\sum_{x\in U_{y_1},z\in U_{y_2}}\sum_{y\in Y}\sum_{u\in U_y}T_{xu}S_{uz} \\
& =n(y_1)^{-\frac{1}{2}}n(y_2)^{-\frac{1}{2}}\sum_{x\in U_{y_1},z\in U_{y_2}}\sum_{y\in Y}n(y)^{-1}\sum_{u,v\in U_y}T_{xu}S_{vz},
\end{align*}
where the third equality uses part (3).  This however, is equal to
\begin{align*}
\sum_{y\in Y}&\Big(n(y_1)^{-\frac{1}{2}}n(y)^{-\frac{1}{2}}\sum_{x\in U_{y_1},u\in U_y}T_{xu}\Big)\Big(n(y_2)^{-\frac{1}{2}}n(y)^{-\frac{1}{2}}\sum_{v\in U_{y_2},v\in U_y}S_{vz}\Big) \\ &=\sum_{y\in Y}\alpha(T)_{y_1y}\alpha(S)_{yy_2}=(\alpha(T)\alpha(S))_{y_1y_2}.
\end{align*}
This implies that $\alpha$ is a $*$-homomorphism.  The fact that $\beta$ defines the inverse for $\alpha$ follows now from more direct computations\footnote{From now on in this section, to keep the length controlled, we will leave such matrix coefficient computations to the reader.} of matrix coefficients, completing the proof of this part.


Finally, for (5), let $T$ be an element of $\C_u[X]$ such that there for each $y\in Y$ there is at most one $x$ such that $y(x)=y$ and $\{T_{xz}~|~z\in X\}$ is not $\{0\}$, and similarly there is at most one $z$ such that $y(z)=y$ and $\{T_{xz}~|~x\in X\}$ is not $\{0\}$.  Define $C,D\in \C_u[X]$ by
$$
C_{xz}=\left\{\begin{array}{ll} 1 & y(x)=y(z) \text{ and } T_{xz'}\neq 0 \text{ for some } z' \\ 0 & \text{ otherwise}\end{array}\right.
$$
and 
$$
D_{xz}=\left\{\begin{array}{ll} \sum_{x'\in X}T_{x'z} & y(x)=y(z) \\ 0 & \text{ otherwise}\end{array}\right.
$$
(note that the sum defining $D$ has at most one non-zero element).  A direct computation shows that $T=CAD$.  As any operator in $\C_u[X]$ can be written as a finite sum of operators $T$ with the properties above, this completes the proof.
\end{proof}

\begin{aprops2}\label{aprops2}
For any non-degenerate representation $\h$ of $\C_u[Y]$ there is a canonically associated non-degenerate representation $\h^X$ of $\C_u[X]$ with the following properties.
\begin{enumerate}
\item The representation $A\cdot \h^X$ of $A\C_u[X]A$ identifies canonically with the representation of $\C_u[Y]$ on $\h$ via the isomorphisms in Lemma \ref{aprops} part (4).
\item If $\h'$ is any non-degenerate representation of $\C_u[X]$ giving rise to a representation $A\cdot \h'$ of $A\C_u[X]A\cong \C_u[Y]$, we have that $\h'$ and $(A\cdot \h')^X$ are canonically isomorphic as $\C_u[X]$ representations.
\end{enumerate}
\end{aprops2}

\begin{proof}
Given $\h$ as in the statement, let $\C_u[X]\odot \h$ denote the algebraic tensor product of $\h$ and $\C_u[X]$, taken over $\C$.  Define a form on this tensor product by the formula
\begin{equation}\label{indform}
\langle S\odot \xi,T\odot \eta \rangle_{\h^X}:=\langle \xi,\alpha(AS^*TA)\eta\rangle_\h
\end{equation}
on elementary tensors, and extending to finite sums of elementary tensors by linearity in the second variable, and conjugate linearity in the first.  This form is clearly linear in the second variable, and conjugate linear in the first. It is also positive semi-definite.  Indeed, note that for any element $\sum_{i=1}^n S_i\odot \xi_i$ of $\C_u[X]\odot\h$, we have that
$$
\Big\langle\sum_{i=1}^n S_i\odot \xi_i,\sum_{i=1}^n S_i\odot \xi_i\Big\rangle=\sum_{i,j=1}^n \langle \xi_i,\alpha(AS_i^*S_jA)\xi_j\rangle.
$$
To show that this is non-negative, it suffices to show that the matrix $(AS_i^*S_jA)_{i,j=1}^n$ is equal to a finite sum of matrices of the form $B^*B$ with $B$ in $M_n(A\C_u[X]A)$.

For each $y\in Y$, then, temporarily write the elements of $U_y$ as $y_1,y_2,...,y_{n(y)}$.  For each $i\in\{1,...,n\}$ and $k\in\{1,...,\max_{y\in Y}n(y)\}$, define $S_i^k$ by
$$
(S_i^k)_{xz}:=\left\{\begin{array}{ll} (S_i)_{xz} &x=y_k \text{ for some } y\in Y \\ 0 & \text{otherwise} \end{array}\right.
$$
and note that 
$$
S_i=\sum_{k=1}^{\max_{y\in Y}n(y)} S_i^k.
$$  
Note moreover that for any $i,j\in\{1,...,n\}$ and $k,l\in\{1,...,\max_{y\in Y}n(y)\}$, $(S_i^k)^*(S_j^l)=0$ unless $k=l$, whence
$$
AS_iS_j^*A=\sum_{k=1}^{\max_{y\in Y}n(y)}A(S_i^k)^*S_j^kA.
$$
It thus suffices to show that for each $k\in\{1,...,\max_{y\in Y}n(y)\}$, the matrix $(A(S_i^k)^*S_j^kA)_{i,j=1}^n$ is of the form $B^*B$ for some $B\in M_n(A\C_u[X]A)$, which we will now do.  Set then 
$$
R_i^k:=ANS_i^kA\in A\C_u[X]A.
$$
Then one checks that $(R_i^k)^*R_j^k=A(S_i^k)^*S_j^k)A$, whence the matrix $(A(S_i^k)^*(S_j^k)A)_{i,j=1}^n$ is equal to
$$
\begin{pmatrix} R_1^k & \hdots & R_n^k \\ 0 & \hdots  & 0 \\ \vdots  & & \vdots \\ 0 & \hdots & 0 \end{pmatrix}^*\begin{pmatrix} R_1^k & \hdots & R_n^k \\ 0 & \hdots  & 0 \\ \vdots  & & \vdots \\ 0 & \hdots & 0 \end{pmatrix}
$$
in $A\C_u[X]A$; this is of the desired form.

Let now $\h^X$ be the corresponding separated completion of $\h\odot \C_u[X]$ for the semi-definite inner product in line \eqref{indform} above.  Let $\h^X_0$ denote the image of $\C_u[X]\odot \h$ in this Hilbert space and write $[S\odot \xi]$ for the class of an element $S\odot \xi\in \C_u[X]\odot \h$ in $\h^X_0$.  For $T\in\C_u[X]$, define an operator $\pi(T)$ on $\h^X_0$ by the formula
$$
\pi(T):\sum_{i=1}^n [S_i\odot \xi_i]\mapsto \sum_{i=1}^n [TS_i\odot \xi_i].
$$
A similar argument to that used above for positivity shows that $\pi(T)$ is bounded, and thus extends to all of $\h^X$.  The map $\pi:\C_u[X]\to\mathcal{B}(\h^X)$ is then clearly a unital $*$-homomorphism, so this gives the desired representation.  We now look at properties (1) and (2).

For property (1), define a linear map  $L:A\cdot\h^X_0\to\h$ by the formula
\begin{equation}\label{ymap}
L:\sum_{i=1}^n [AS_i\odot \xi_i]\mapsto \sum_{i=1}^n \alpha(AS_iA)\xi_i,
\end{equation}
and note that
\begin{align*}
\Big\langle \sum_{i=1}^n [AS_i\odot \xi_i],\sum_{i=1}^n [AS_i\odot \xi_i]\Big\rangle_{\h^X} & =\sum_{i,j=1}^n\langle \xi_i,\alpha(AS_i^*AAS_jA)\xi_j\rangle_\h \\
&=\sum_{i,j=1}^n \langle \alpha(AS_iA)\xi_i,\alpha(AS_jA)\xi_j\rangle \\ &=\Big\langle \sum_{i=1}^n\alpha(AS_iA)\xi_i,\sum_{i=1}^n\alpha(AS_iA)\xi_i\Big\rangle_\h.
\end{align*}
This implies that $L$ as in line \eqref{ymap} is an isometry from $A\cdot \h^X_0$ to $\h$, so extends to an isometric map, which is clearly onto by non-degeneracy.  It is also clear that $L$ intertwines the representations of $A\C_u[X]A\cong \C_u[Y]$.

Finally, we look at property (2).  Define a map $M:(A\cdot \h')_0^X\to \h'$ by the formula
$$
M:\sum_{i=1}^n [S_i\odot A\xi_i]\mapsto \sum_{i=1}^n S_iA\xi_i.
$$
Computing,
\begin{align*}
\Big\langle \sum_{i=1}^n [S_i\odot A\xi_i],\sum_{i=1}^n [S_i\odot A\xi_i]\Big\rangle_{(A\cdot \h')^X} &=\sum_{i,j=1}^n \langle \xi_i,\beta(\alpha(AS_i^*S_jA))\xi_j\rangle_{\h'} \\ &= \Big\langle \sum_{i=1}^n S_iA\xi_i,\sum_{i=1}^n S_iA\xi_i\Big\rangle_{\h'}.
\end{align*}
This implies that $M$ again extends to an isometric linear map, and Lemma \ref{aprops} part (5) and non-degeneracy implies that this is onto.  Again, it clearly intertwines the representations of $\C_u[X]$, so the proof is complete.
\end{proof}

It follows from the lemma above that non-degenerate $\C_u[X]$ and $\C_u[Y]\cong A\C_u[X]A$ representations come canonically in pairs $(\h^X,\h^Y)$ such that $A\cdot \h^X=\h^Y$.  We will make these assumptions (and use this notation) throughout the rest of the proof of Proposition \ref{inc}.

Our next task is to study the relationship between the constant vectors $\h^X_c$ and $\h^Y_c$ in the spaces above.  Note that Lemma \ref{aprops} part (1) implies that $\h^X_c$ is a subspace of $\h^Y$ (and that $\h^Y_c$ is a subspace of $\h^Y$ by definition).  Let $\Phi_X:\C_u[X]\to l^\infty(X)$ and $\Phi_Y:\C_u[Y]\to l^\infty(X)$ be the linear maps defined in line \eqref{phidef} above and define 
\begin{equation}\label{psiy}
\Psi_Y:=\beta\circ \Phi_Y\circ \alpha:A\C_u[X]A\to \beta(l^\infty(Y))\subseteq A\C_u[X]A.
\end{equation}
Note that Lemma \ref{conslem} implies that a vector $\xi\in \h^Y$ is in $\h^Y_c$ if and only if $\Psi_Y(ATA)\xi=ATA\xi$ for all $T\in \C_u[X]$.  The following computations contain the bulk of the rest of the proof of Proposition \ref{inc}.

\begin{conn}\label{conn}
The following hold.
\begin{enumerate}
\item For any partial translation $v\in\C_u[Y]$
$$
\Phi_X(N\beta(v)N^{-1})A=\beta(vv^*)A.
$$
\item For any partial translation $v\in \C_u[X]$ such that for all $y\in Y$
\begin{equation}\label{smallsupp}
|\{x\in U_y~|~(vv^*)_{xx}=1\}|\leq 1 \text{ and } |\{z\in U_y~|~(v^*v)_{zz}=1\}|\leq 1
\end{equation}
we have the formula
$$
\Psi_Y(N^{-1}AvAN)=Avv^*A.
$$
\item The operator $N$ from line \eqref{bign} above on $\h^X$ restricts to an isomorphism 
$$
N:\h^Y_c\to \h^X_c.
$$
\item For any $\xi\in (\h^X_c)^\perp\cap \h^Y$, if we decompose $N^{-1}\xi=\xi_1+\xi_2$ where $\xi_1\in \h^Y_c$ and $\xi_2$ is in $(\h^Y_c)^\perp$, then 
$$
\|\xi_2\|\geq \frac{\|\xi\|}{\sqrt{1+(\|N\|\|N^{-1}\|)^2}}.
$$
\item For any $\xi\in (\h^Y_c)^\perp\cap \h^Y$, if we decompose $N\xi=\xi_1+\xi_2$ where $\xi_1\in\h^X_c$ and $\xi_2$ is in $(\h^X_c)^\perp$, then 
$$
\|\xi_2\|\geq \frac{\|\xi\|}{\sqrt{1+(\|N\|\|N^{-1}\|)^2}}.
$$
\end{enumerate}
\end{conn}

\begin{proof}
For part (1), direct computations show that for any $x,z\in X$, the corresponding matrix coefficients are given by
$$
(\beta(vv^*)A)_{xz}=(\Phi_X(N\beta(v)N^{-1})A)_{xz}=n(x)^{-1}
$$
if $y(x)=y(z)$ and $(vv^*)_{y(x)y(x)}=1$ and zero otherwise.  For part (2), a direct computation using the formulas in lines \eqref{phidef},  \eqref{alpha} and \eqref{beta} shows that for any $T\in A\C_u[X]A$, the matrix coefficients of $\Psi_Y(T)$ are given by
$$
(\Psi_Y(T))_{xz}=n(x)^{-\frac{3}{2}}\sum_{y\in Y}\sum_{x'\in U_{y(x)},~z'\in U_y}n(z')^{-\frac{1}{2}}T_{x'z'}.
$$
From here, more direct computation shows that the matrix coefficients of the operators in the statement are given by
$$
(Avv^*A)_{xz}=(\Psi_Y(N^{-1}AvAN))_{xz}=n(x)^{-2}
$$
if $y(x)=y(z)$ and there exists $x'\in U_{y(x)}$ such that $(vv^*)_{x'x'}=1$, and zero otherwise.

For part (3), assume first that $\xi$ is an element of $\h_c^X$; we want to show that $N^{-1}\xi$ is in $\h_c^Y$ and thus that $N\cdot \h_c^Y$ is a superspace of $\h_c^X$.  Let $v\in \C_u[Y]$ be an arbitrary  partial translation; we want to show that $\beta(v)N^{-1}\xi=\beta(vv^*)N^{-1}\xi$.  Then using that $\xi$ is in $\h^X_c\subseteq \h^Y$, we have
$$
N\beta(v)N^{-1}\xi=\Phi_X(N\beta(v)N^{-1})\xi=\Phi_X(N\beta(v)N^{-1})A\xi.
$$
Using part (1), this is equal to
$$
\beta(vv^*)A\xi=N\beta(vv^*)N^{-1}\xi
$$
using that $N$ commutes with $\beta(vv^*)$.  Hence $N\beta(v)N^{-1}\xi=N\beta(vv^*)N^{-1}\xi$ and cancelling the $N$ gives the desired conclusion.

Conversely, assume that $\xi$ is an element of $\h_c^Y$; we want to show that $N\xi$ is in $\h_c^X$ and thus that $N\cdot \h^Y_c$ is a subspace of $\h_c^X$.  Let $v\in \C_u[X]$ be a partial translation; we want to show $vv^*N\xi=vN\xi$.  Splitting $v$ up as a finite sum of at most $(\max_{x\in X}n(x))^2$ elements, we may assume that $v$ satisfies the conditions in line \eqref{smallsupp} for any $y\in Y$.  Let now $C\in \C_u[X]$ be defined by
\begin{equation}\label{collapse}
C_{xz}=\left\{\begin{array}{ll} (vv^*)_{xx} & y(x)=y(z) \\ 0 & \text{ otherwise}\end{array}\right.
\end{equation}
(roughly speaking, $C$ collapses each $U_y$ that intersects the range of $v$ into the single point in which it intersects the range of $v$).  Note that $N$ commutes with $C$.  We then have the formula $CAv=v$.  Now,
\begin{align*}
N^{-1}vN\xi=N^{-1}CAvN\xi=CN^{-1}AvAN\xi=C\Psi_Y(N^{-1}AvAN)\xi
\end{align*}
where the last equality uses that $\xi$ is in $\h_c^Y$.  Continuing using part (2), this is equal to 
$$
CAvv^*A\xi=vv^*\xi=N^{-1}vv^*N\xi,
$$
where the second equality uses that $N$ commutes with $vv^*$ by the assumption in line \eqref{smallsupp}.  Hence $N^{-1}vv^*N\xi=N^{-1}vN\xi$ so $vv^*N\xi=vN\xi$ as required.

For part (4), note that if $\xi=N\xi_1+N\xi_2$ and $N\xi_1$ is in $\h^X_c$ by part (3).  Hence taking the inner product with $N\xi_1$ gives
$$
0=\|N\xi_1\|^2+\langle N\xi_1,N\xi_2\rangle,
$$
whence
$$
\|N\xi_1\|^2=|\langle N\xi_1,N\xi_2\rangle|\leq \|N\xi_1\|\|N\xi_2\|
$$
and so (assuming as we may that $N\xi_1\neq 0$), $\|N\xi_1\|\leq \|N\xi_2\|$.  This in turn implies that
$$
\frac{\|\xi_1\|}{\|N^{-1}\|}\leq \|N\xi_1\|\leq \|N\xi_2\|\leq \|N\|\|\xi_2\|,
$$
so 
$$
\|\xi_1\|\leq \|N\|\|N^{-1}\|\|\xi_2\|.
$$
This combined with the fact that $\|\xi_1\|^2+\|\xi_2\|^2=\|\xi\|^2$ forces
$$
\|\xi_2\|^2(1+(\|N\|\|N^{-1}\|)^2)\geq \|\xi\|^2,
$$
from which the claimed inequality follows.  Part (5) is analogous, and we are done.
\end{proof}

We are now finally ready to complete the proof of Proposition \ref{inc}, and thus also of Theorem \ref{cinv}.

\begin{proof}[Proof of Proposition \ref{inc}]
Assume first that $Y$ has geometric property (T).  Let $F$ be a generating controlled set for the coarse structure on $Y$, and let $E$ be any generating controlled set for the coarse structure on $X$ that contains $F$, the controlled set $\{(x,z)\in X\times X~|~y(x)=y(z)\}$ for $X\times X$, and the composition of these two.  Using Proposition \ref{nclem}, it will suffice to show that there exists some constant $\epsilon>0$ depending only on $F$ and the cover $\{U_y\}$ such that for any unit vector $\xi\in (\h^X_c)^\perp$ there exists $T\in \C_u[X]$ supported in $E$ with matrix coefficients bounded by a number depending only on $\{U_y\}$ and $F$, and with $\|(T-\Phi(T))\xi\|\geq \epsilon$.

Let $c\in (0,1)$ be a constant, which will be chosen later in a way that depends only on the cover $\{U_y\}$ of $X$ and $F$.  Note that if $\|(1-A)\xi\|\geq c$, then we are done, as $\Phi(A)=1$.  Assume then that 
\begin{equation}\label{thatc}
\|(1-A)\xi\|\leq c.
\end{equation}

Now, by part (4) of Lemma \ref{conn} (and the fact that $A$ preserves $(\h_c^X)^\perp$) we may write $N^{-1}A\xi=\xi_1+\xi_2$ where $\xi_1$ is in $\h_c^Y$, $\xi_2$ is in $(\h_c^Y)^\perp$, and $\|\xi_2\|\geq c_1\|A\xi\|$ for some $c_1>0$ depending only on the cover $\{U_y\}$.  Using geometric (T) for $Y$ (and the fact that $\xi_1$ is in $\h^Y_c$) there exists a partial translation $v\in \C_u[Y]$ supported in $F$ and a constant $c_2>0$ depending only on $F$ such that  such that
$$
\|(\beta(v)-\beta(vv^*))N^{-1}A\xi\|=\|(\beta(v)-\beta(vv^*))\xi_2\|\geq c_2\|\xi_2\|\geq c_2c_1\|A\xi\|.
$$
Hence
$$
\|N(\beta(v)-\beta(vv^*)N^{-1})A\xi\|\geq \frac{c_1c_2}{\|N^{-1}\|}\|A\xi\|.
$$
Now, using Lemma \ref{conn} part (1) and the fact that $N$ commutes with $\beta(vv^*)$, this implies that
$$
\|N\beta(v)N^{-1}-\Phi_X(N\beta(v)N^{-1}))A\xi\|\geq c_3\|A\xi\|,
$$
where $c_3>0$ depends only $\{U_y\}$ and $F$ again.  Finally, this forces
\begin{align*}
\|N\beta(v)N^{-1}&-\Phi_X(N\beta(v)N^{-1}))\xi\| \\ & \geq c_3-\|N\beta(v)N^{-1}-\Phi_X(N\beta(v)N^{-1}))(1-A)\xi\| \\ & \geq c_3-\|N\beta(v)N^{-1}-\Phi_X(N\beta(v)N^{-1})\|c,
\end{align*}
where $c$ is as in line \eqref{thatc}.  Noting that 
$$\|N\beta(v)N^{-1}-\Phi_X(N\beta(v)N^{-1}))\|=\|N(\beta(v)-\beta(vv^*)N^{-1})\|\leq 2\|N\|\|N^{-1}\|,$$ setting $c=\frac{c_3}{4\|N\|\|N^{-1}\|}$ see see that 
$$
\|N\beta(v)N^{-1}-\Phi_X(N\beta(v)N^{-1}))\xi\|\geq  \frac{c_3}{2}.
$$
Note that $N\beta(v)N^{-1}$ is supported in $F\circ \{(x,z)\in X\times X~|~y(x)=y(z)\}$.

In summary, we have shown the desired conclusion with $\epsilon=\min\{c,c_3/2\}$: if $\|(1-A)\xi\|\geq c$, we may take $T=A$, and otherwise we may take $T=N\beta(v)N^{-1}$.\\

For the converse implication, assume $X$ has geometric property (T), and let $F$ be a controlled generating set for $X$.  Let $E$ be any  controlled set for $X$ that contains a generating set for $Y$ and is such that
$$F\circ \{(x,z)\in X\times X~|~y(x)=y(z)\}$$
Using Proposition \ref{nclem} and the isomorphism $A\C_u[X]A\cong \C_u[Y]$ from Lemma \ref{aprops} part (4), it will suffice to show that there exists some constant $\epsilon>0$ depending only on $F$ and the cover $\{U_y\}$ such that for any unit vector $\xi\in (\h^Y_c)^\perp$ there exists $T\in A\C_u[X]A$ supported in $E$ with matrix coefficients bounded by some number depending only on $F$ and $\{U_y\}$, and such that $\|(T-\Phi(T))\xi\|\geq \epsilon$.

Let $\xi$ be a unit vector in $(\h_c^Y)^{\perp}\cap \h^Y$.  Using Lemma \ref{conn} part (5) we may write $N\xi=\xi_1+\xi_2$ where $\xi_1$ is in $\h^X_c$, $\xi_2$ is in $(\h_c^X)^\perp$ and $\|\xi_2\|\geq c_1\|\xi\|$ for some $c_1>0$ depending only on the cover $\{U_y\}$.  

Using geometric property (T), there exists a partial translation $v\in \C_u[X]$ supported in $F$ and $c_2>0$ depending only on $F$ such that 
$$
\|(vv^*-v)N\xi\|=\|(vv^*-v)\xi_2\|\geq c_2\|\xi_2\|\geq c_1c_2.
$$
We may split $v$ up as a finite sum of at most $\max_{x\in X}n(x)^2$ partial translations satisfying the support condition in line \eqref{smallsupp}, and thus assume that 
\begin{equation}\label{midinq}
\|(vv^*-v)N\xi\|\geq c_3
\end{equation}
where $c_3=c_1c_2/(\max_{x\in X}n(x)^2)$ and $v$ satisfies the support condition in line \eqref{smallsupp}.  

Now, let $C$ be the `collapsing' operator defined as in line \eqref{collapse} above for this $v$.  Using that $CAv=v$ and $A\xi=A$, we see that
$$
\|(vv^*-v)N\xi\|=\|C(Avv^*-Av)N\xi\|=\|C(Avv^*A-AvA)N\xi\|.
$$
As $A\xi=\xi$ and $N$ commutes with $A$, $C$ and $vv^*$, this implies that
$$
\|C(Avv^*A-N^{-1}AvAN)\xi\|\geq \frac{c_2}{\|N\|}.
$$
Hence by Lemma \ref{conn} part (2), we see that
$$
\|C(\Psi_Y(N^{-1}AvAN)-N^{-1}AvAN)\xi\|\geq \frac{c_2}{\|N\|},
$$
so, as $A$ commutes with $N$,
$$
\|(\Psi_Y(AN^{-1}vNA)-AN^{-1}vNA)\xi\|\geq \frac{c_2}{\|N\|\|C\|}.
$$
As $\|C\|$ admits an upper bound $c_3$ depending only on the cover $\{U_y\}$, this completes the proof: take $\epsilon=c_2/\|N\|c_3$ and $T=AN^{-1}vNA$.
\end{proof}

The following corollary gives our first examples of spaces with geometric property (T); in some sense these could be considered `trivial' examples.

\begin{fin}\label{fin}
Let $X$ be a space which splits into coarse components $X=\sqcup X_n$ such that $\max_{n\in \N}|X_n|$ is finite.  Then $X$ has geometric property (T).
\end{fin}

\begin{proof}
If each $X_n$ is a single point, then for any representation $\h$ of $\C_u[X]$ we have $\h=\h_c$, so geometric property (T) is trivially satisfied.  Any space $X$ as in the statement is coarsely equivalent to such a space where each $X_n$ is a single point, however.
\end{proof}

\section{Laplacians}\label{lapsec}

In this section, we define Laplacian operators, and use them to give another characterisation of geometric property (T).  This characterisation in terms of Laplacians was our original definition of geometric property (T) in \cite[Section 7]{Willett:2010zh}, and is more closely connected to $K$-theory.  It also lets us relate geometric property (T) to expanding graphs. 

Throughout this section, $X$ denotes a space as in  Definition \ref{spac}.

\begin{lapdef}\label{lapdef}
Let $E$ be a controlled set for $X$.  The \emph{Laplacian associated to $E$}, denoted $\Delta^E$, is the element of $\C_u[X]$ with matrix coefficients defined by
$$
\Delta^E_{xy}=\left\{\begin{array}{ll} -1 & (x,y)\in (E\cup E^{-1})\setminus \dg(E)  \\ |\{z\in X~|~(x,z)\in (E\cup E^{-1})\backslash \dg(E)\}| &x=y \\ 0 & \text{otherwise}\end{array}\right.
$$

\end{lapdef}
Note that $\Delta^E$ only depends on $(E\cup E^{-1})\backslash \dg(E)$.  Note also that if $E$ is empty, or a subset of the diagonal, then $\Delta^E$ is $0$. 

\begin{lapex}\label{lapex}
Say $X$ is the vertex set of an undirected graph, with the coarse structure generated by the subset $E$ of $X\times X$ consisting of all the edges as in Example \ref{spacelem}.  The (un-normalised) combinatorial Laplacian of $X$ in the sense of spectral graph theory (see for example \cite[Section 4.2]{Lubotzky:1994tw}) is then the same as our $\Delta^E$.  This is the motivating example.
\end{lapex}

The next two lemmas record some basic properties of Laplacians associated to antisymmetric elementary controlled sets (see Definition \ref{ops} for the terminology).

\begin{poslem}\label{poslem}
Let $E$ be an antisymmetric elementary controlled set, and $\Delta^E$ the corresponding Laplacian.  Let $t:A\to B$ be the partial translation\footnote{Recall that $E$ being antisymmetric means that $A\cap B=\varnothing$.} such that $E=\text{graph}(t)$ and $v$ the partial translation operator corresponding to $t$ .
\begin{enumerate}
\item $\Delta^E$ and $v$ are related by the equation.
$$
\Delta^E=vv^*+v^*v-v-v^*.
$$
\item The image of $\Delta^E$ in any $*$-representation is a positive operator.
\item If $\h$ is any $*$-representation of $\C_u[X]$, then the kernel of $\Delta^E$ consists precisely of those vectors $\xi\in \h$ such that
$$
v\xi=vv^*\xi ~~(\text{equivalently, such that } v^*\xi=v^*v\xi).
$$
\end{enumerate}
\end{poslem}

\begin{proof}
For the first part, one checks directly that for both operators $\Delta^E$ and $vv^*+v^*v-v-v^*$, the $(x,y)^\text{th}$ matrix coefficient is equal to
$$
\left\{\begin{array}{ll} 
-1 & x\neq y, \text{ and either } t(x)=y \text{ or } t(y)=x \\
1 & x=y, \text{ and }x\in A\cup B \\
0 & \text{otherwise}.\end{array}\right. 
$$
The second and third parts both follow from the formula
$$
\Delta^E=vv^*+v^*v-v-v^*=(vv^*-v)^*(vv^*-v).  \eqno\qedhere
$$
\end{proof}

\begin{poscor}\label{poscor}
Let $E$ be an elementary controlled set such that 
$$
(E\cap E^{-1})\setminus \dg(E)=\varnothing.
$$ 
Let $v$ the corresponding partial translation operator, and $\Delta^E$ the corresponding Laplacian.  Then in any $*$-representation $\h$ of $\C_u[X]$, the kernel of $\Delta^E$ consists precisely of those vectors $\xi\in \h$ such that
$$
v\xi=vv^*\xi ~~(\text{equivalently, such that } v^*\xi=v^*v\xi).
$$
\end{poscor}

\begin{proof}
Let $t:A\to B$ be the partial translation underlying $E$.  Using Lemma \ref{dislem}, we may decompose 
$$
A=A_0\sqcup A_1\sqcup A_2\sqcup A_3
$$
such that for $i\in\{0,1,2\}$, $t(A_i)\cap A_i=\varnothing$, and so that the restriction of $t$ to $A_3$ is the identity.  Write $t_i$ for the restriction of $t$ to $A_i$, $E_i$ for the graph of $t_i$, and $v_i$ for the corresponding partial translation operator.  The condition 
$$
(E\cap E^{-1})\setminus \dg(E)=\varnothing
$$
on $E$ implies that we have a \emph{disjoint} union
$$
(E\cup E^{-1})\setminus \dg(E)=(E_0\sqcup E_0^{-1})\sqcup (E_1\sqcup E_1^{-1})\sqcup (E_2\sqcup E_2)^{-1},
$$
which implies by a direct computation of matrix coefficients that  
$$
\Delta^E=\Delta^{E_0}+\Delta^{E_1}+\Delta^{E_2}.
$$
Part (2) of Lemma \ref{poslem} implies that all the operators $\Delta^{E_i}$ are positive, and combining this with part (3) of Lemma \ref{poslem}, we have that in any $*$-representation $\h$ of $\C_u[X]$,
$$
\text{Kernel}(\Delta^E)=\bigcap_{i=0}^2\text{Kernel}(\Delta^{E_i})=\{\xi\in \h ~|~v_i\xi=v_iv_i^*\xi \text{ for all } i\in \{0,1,2\}\}
$$
On the other hand, the facts that $v_0,v_1,v_2,v_3$ have mutually orthogonal domains and mutually orthogonal ranges, and that $v_3^*v_3=v_3^*=v_3$ imply that
$$
v=v_0+v_1+v_2+v_3 ~~\text{ and }~~vv^*=v_0v_0^*+v_1v_1^*+v_2v_2^*+v_3v_3^*,
$$
and moreover that the condition `$v_i\xi=v_iv_i^*\xi \text{ for all } i\in \{0,1,2\}$' on vectors in $\h$ is equivalent to `$vv^*\xi=v\xi$', so we are done.
\end{proof}

\begin{laplem}\label{laplem}
If $E\subseteq F$ are controlled sets then there exist antisymmetric elementary controlled sets $E_1,...,E_n$ such that
$$
\Delta^F=\Delta^E+\sum_{i=1}^n\Delta^{E_i}.
$$

In particular, in any $*$-representation of $\C_u[X]$ we have the operator inequality
$$
\Delta^F\geq \Delta^E\geq 0.
$$
\end{laplem} 

\begin{proof}
Lemma \ref{elde} implies that there exist antisymmetric elementary controlled sets $E_1,...,E_n$ such that $(F\cup F^{-1})\backslash\dg(F)$ can be written as the disjoint union 
$$
(F\cup F^{-1})\backslash\dg(F)=\big((E\cup E^{-1})\backslash \dg(E)\big)\sqcup \bigsqcup_{i=1}^n (E_i\sqcup E_i^{-1})
$$
It follows by a direct computation of matrix coefficients that 
$$
\Delta^F=\Delta^E+\sum_{i=1}^n\Delta^{E_i}.
$$
The operator inequality $\Delta^F\geq \Delta^E$ now follows from positivity of each $\Delta^{E_i}$ as in part (2) of Lemma \ref{poslem}.  The fact that $\Delta^E\geq 0$ for any controlled set $E$ follows from the special case inclusion $\varnothing\subseteq E$.
\end{proof}

\begin{conlem}\label{conlem}
Let $E$ be a controlled set and $\mathcal{H}$ a representation of $\C_u[X]$.   The constant vectors $\h_c$ in $\mathcal{H}$ are contained in the kernel of $\Delta^E$.  If moreover $E$ is generating, then the kernel of $\Delta^E$ is precisely equal to $\h_c$.
\end{conlem}

\begin{proof}
Assume first that $E$ is a general controlled set.  Lemma \ref{laplem} implies that there are antisymmetric elementary controlled sets $E_1,...,E_n$ such that 
$$
\Delta^E=\sum_{i=1}^n \Delta^{E_i}.
$$
Letting $v_i$ be the partial translation operator corresponding to $E_i$, part (3) of Lemma \ref{poslem} implies that the kernel of $\Delta^{E_i}$ consists precisely of those $\xi\in \h$ such that $v_i\xi=v_iv_i^*\xi$, and thus contains $\h_c$.  On the other hand, 
$$
\text{Kernel}(\Delta^E)= \bigcap_{i=1}^n\text{Kernel}(\Delta^{E_i}),
$$
whence $\text{Kernel}(\Delta^E)\supseteq \h_c$.

Assume now that $E$ is generating and that $\Delta^E\xi=0$.  Let $v$ be a partial translation operator; we must show that $v\xi=vv^*\xi$.  Say $v$ corresponds to the partial translation $t:A\to B$.  As $E$ is generating there exists $n$ such that $E^{\circ n}$ contains $\text{graph}(t)$, whence Lemma \ref{ptde} implies that there exists $m$ and a decomposition $A=A_1\sqcup\cdots\sqcup A_m$ such that if $t_i:=t|_{A_i}$ then there exist partial translations $s_i^1,...,s_i^n$ such that
\begin{itemize}
\item $t_i=s_i^n\circ \cdots \circ s_i^1$;
\item $\text{graph}(s_i^j)\subseteq E$ for all $i=1,...,m$ and $j=1,...,n$;
\item for each $i$ and each $j=1,...,n-1$, the range of $s_i^j$ is equal to the domain of $s_i^{j+1}$;
\item for each $i$ and each $j=1,...,n-1$, either $s_i^j$ is the identity map, or the range of $s_i^{j+1}$ is disjoint from its support.
\end{itemize}
Let $v_i^j$ be the operator corresponding to $s_i^j$, and $v_i$ the operator corresponding to $t_i$.  For fixed $i,j$, let 
$$
F=\text{graph}(s^j_i).
$$
Then $0\leq \Delta^F\leq \Delta^E$ by Lemma \ref{laplem} whence $\Delta^F\xi=0$.  Corollary \ref{poscor} then implies that $v_i^j\xi=v^j_i(v^j_i)^*\xi$ (and this is true for all $i,j$, as the choice of indices was arbitrary).  

To complete the proof, assume inductively for some $i$ and $j=1,...,n-1$ that if $u:=v^j_iv_i^{j-1}\cdots v_i^1$ then $uu^*\xi=u\xi$.  Then as the support of $w:=v^{j+1}_i$ is the range of $u$ we have
$$
wu\xi=wuu^*\xi=w\xi=ww^*\xi=(wu)(wu)^*\xi,
$$
whence by induction $v_i\xi=v_iv_i^*\xi$ for each $i$.  Finally, note that
$$
v\xi=(v_1+\cdots +v_n)\xi=(v_1v_1^*+\cdots+v_nv_n^*)\xi;
$$
this, however, is equal to $vv^*\xi$ using that the operators $v_i$ all have orthogonal ranges, and we are done.
\end{proof}

\begin{spectra0}\label{spectra0}
Let $T$ be an element of $\C_u[X]$ and $\h$ a representation of $\C_u[X]$.  Define $\sigma_\h(T)$ to be the spectrum of $T$ considered as an operator on $\h$ via this representation.  

Define the \emph{maximal spectrum} of $T$, $\sigma_{max}(T)$, to be the union of all the sets $\sigma_\h(T)$ as $\h$ ranges over all representations of $\C_u[X]$.
\end{spectra0}

We are now ready to relate geometric property (T) to Laplacians.

\begin{tlap}\label{tlap}
The following are equivalent.
\begin{enumerate}
\item $X$ has geometric property (T).
\item For any controlled set $E$ there exists $c=c(E)>0$ such that $\sigma_{max}(\Delta^E)\subseteq \{0\}\sqcup [c,\infty)$.
\item For some controlled set $E$ there exists $c>0$ such that $\sigma_{max}(\Delta^E)\subseteq \{0\}\sqcup [c,\infty)$.
\end{enumerate}
\end{tlap}

\begin{proof}
We will only prove that (1) and (2) are equivalent: one can show that (3) is equivalent to conditions (2) and (4) from Proposition \ref{nclem} analogously.  

Assume first that $X$ satisfies condition (2).  Noting that $\Phi(\Delta^E)=0$ for any controlled set $E$ and using Proposition \ref{conlem}, it is clear that $X$ then satisfies condition (3) from Proposition \ref{nclem} with $T=\Delta^E$.  

Assume conversely that $X$ has geometric property (T), let $E$ be a controlled set, and $c=c(E)>0$ be as in the definition of geometric property (T).  Let $\h$ be a representation of $\C_u[X]$, and $\xi\in \h_c^\perp$ a unit vector.  Let $v$ be a partial translation with support in $E$ such that $\|(vv^*-v)\xi\|\geq c$, which exists by geometric property (T).  Lemma \ref{dislem} implies that we may write $v=v_0+v_1+v_2$ where each $v_i$ is a partial translation corresponding to an antisymmetric elementary controlled set, and the $v_i$ have mutually orthogonal ranges.  It follows that from the orthogonality of the ranges that
$$
\sum_{i=0}^2\|(v_iv_i^*-v_i)\xi\|^2=\|(vv^*-v)\xi\|^2\geq c^2,
$$
whence for some $i$, $\|(v_iv_i^*-v_i)\xi\|\geq c/\sqrt{3}$.  In particular, on altering the constant $c$ and replacing $v$ by one of the $v_i$, we may assume that $v$ comes from an antisymmetric elementary set.

Now by Lemmas \ref{poslem} and \ref{laplem}, we may write
$$
\Delta^E=\sum_{i=1}^nv_iv_i^*+v_i^*v_i-v_i-v_i^*
$$
where $v_1=v$, and the other $v_i$ are partial translations with support in $E$.  Using Lemma \ref{poslem} again (and its proof), it follows that
$$
\langle \xi,\Delta^E\xi\rangle\geq \langle \xi,(vv^*+v^*v-v-v^*)\xi\rangle=\|(vv^*-v)\xi\|^2\geq c^2;
$$
as $\xi$ was an arbitrary element of $\h_c^\perp$, and $\h$ was itself arbitrary, this shows that $\sigma_{max}(\Delta^E)$ is contained in $\{0\}\sqcup[c^2,\infty)$, so we are done.
\end{proof}

Our work on Laplacians allows us to give an easy proof of the following consequence of geometric property (T) for sequences of graphs: it implies that the sequence of graphs is an \emph{expander} in the sense of the following condition.

\begin{expa}\label{expa}
Let $(X_n)$ be a sequence of (vertex sets of) finite connected graphs.  The sequence $(X_n)$ is an \emph{expander} if the following hold
\begin{enumerate}[(i)]
\item the cardinalities $|X_n|$ tend to infinity;
\item there is a uniform bound on the degrees of all vertices in each $X_n$;
\item there exists some $c>0$ such that if $\Delta_n$ is the graph Laplacian on $l^2(X_n)$ as in Example \ref{lapex}, then the spectrum of $\Delta_n$ is contained in $\{0\}\sqcup [c,\infty)$.
\end{enumerate}
\end{expa}

Expanders have applications in several areas of pure mathematics, as well as computer science and information theory: see \cite{Lubotzky:1994tw} for more information.

\begin{fin2}\label{fin2}
Let $X$ be a space that decomposes into coarse components as $X=\sqcup X_n$, and assume that $|X_n|$ is finite and that $|X_n|$ tends to infinity.  Let $E$ be a symmetric generating set for the coarse structure.  Define a (connected) graph structure on each $X_n$ by decreeing that $E\cap (X_n\times X_n)$ is the edge set, and call the corresponding graph $G_n$. 

Then the sequence $(X_n)$ is an expander.
\end{fin2}

\begin{proof}
Let $\Delta_n$ be the graph Laplacian on each $X_n$.  Then $(X_n)$ is an expander if and only if the operator 
$$
\Delta:=\oplus \Delta_n\in \mathcal{B}(\oplus_i l^2(X_n))
$$
has spectrum contained in some set of the form $\{0\}\sqcup [c,\infty)$.  This follows, however, as $\Delta$ identifies with $\Delta^E$ acting on $l^2(X)$, so the spectrum of $\Delta$ is equal to $\sigma_{l^2(X)}(\Delta^E)$, this is contained in $\sigma_{max}(\Delta^E)$, and this is a subset of $\{0\}\sqcup [c,\infty)$ by geometric property (T).
\end{proof}

\begin{exrem}\label{exrem}
The methods of Section \ref{cinvsec} can be used to show that `being an expander' is a coarse invariant of a sequence of graphs in the obvious sense.  Although known to some experts\footnote{It also admits a rather easier proof, as pointed out to us by Romain Tessera.}, this does not seem to have been observed in the literature before.
\end{exrem}

\section{Relationship with amenability}\label{amensec}

In this section we discuss the relationship between geometric property (T) and amenability.  Throughout the section, $X$ denotes a space.

The main result is Proposition \ref{amen}; phrased slightly differently, it says that $\C_u[X]$ admits a representation $\h$ where the space $\h_c$ of constant vectors is non-zero if and only if $X$ is amenable.  It follows (Corollary \ref{infi}) that for coarsely connected spaces, geometric property $T$ is equivalent to non-amenability.

We start with a lemma.  Variants of this are very well-known, but we include a proof for the reader's convenience, and as we could not find exactly what we needed in the literature.

\begin{amenlem}\label{amenlem}
The following are equivalent.
\begin{enumerate}
\item There exists an \emph{invariant mean} on $X$: a positive unital linear functional
$$
\phi:l^\infty(X)\to\C
$$
such that if $f\in l^\infty(X)$, and $t:A\to B$ is any partial translation such that $B$ contains the support of $f$, then 
$$
\phi(f)=\phi(f\circ t).
$$
\item There exists a net $(\xi_i)_{i\in I}$ of unit vectors in $l^2(X)$ such that for any partial translation $v$.
$$
\lim_{i\in I}\|v\xi_i-vv^*\xi_i\|=0.
$$
\end{enumerate}
\end{amenlem}

\begin{proof}
Assume first condition (1).  Fix a finite set $\{v_1,...,v_m\}$ of partial translations.  It suffices to show that there exists a sequence of unit vectors $(\xi_n)$ in $l^2(X)$ such that  
$$
\lim_{n\to\infty}\|v_i\xi_n-v_iv_i^*\xi_n\|=0
$$
for all $i=1,...,m$.  
Let $P(X)$ denote the space of finitely supported probability measures on $X$ (a subset of $l^1(X)$), which identifies with a weak-$*$ dense subset of the space of positive unital linear functionals on $l^\infty(X)$ via the standard pairing between $l^1$ and $l^\infty$.  Let $(\phi_j)_{j\in J}$ be a net in $P(X)$ that converges weak-$*$ to $\phi$, and for each $i=1,...,m$, let $t_i:A_i\to B_i$ be the partial bijection corresponding to $v_i$.  Then for each $i$ and any $f\in l^\infty(B_i)$ we have
$$
\lim_{j\in J}(\phi_j(f)-\phi_j(f\circ t_i))=0,
$$
or in other words that 
$$
\phi_j|_{B_i}-(\phi_j\circ t_i^{-1})|_{A_i}
$$
converges weakly to zero in $l^1(X)$, whence $0$ is in the weak closure of the convex set
$$
\Big\{\oplus_{i=1}^m (\psi|_{B_i}-(\psi\circ t_i^{-1})|_{A_i})\in \bigoplus_{i=1}^m l^1(X)~\Big|~\psi\in P(X)\Big\}.
$$
The Hahn-Banach theorem thus implies it is in the norm closure, i.e.\ there is a sequence $(\phi_n)$ of elements of $P(X)$ such that 
\begin{equation}\label{lim1}
\lim_{n\to\infty}\|\phi_n|_{B_i}-(\phi_n\circ t_i^{-1})|_{A_i}\|_1=0
\end{equation}
for $i=1,...,m$.  Set $\xi_n(x)=\sqrt{\phi_n(x)}$ for each $n$ and $x\in X$, so each $\xi_n$ is a unit vector in $l^2(X)$. For any $i=1,...,m$, we have
$$
\|v_i\xi_n-v_iv_i^*\xi_n\|_2=\|(\xi_n\circ t_i^{-1})|_{A_i}-\xi_n|_{B_i}\|_2\leq 2\|\phi_n|_{B_i}-(\phi_n\circ t_i^{-1})|_{A_i}\|_1,
$$
which tends to zero as $n$ tends to infinity.

For the converse, let $(\xi_i)$ be a net with the properties given.  Then it is not difficult to check that any weak-$*$ limit point of the functionals
$$
\phi_i:f\mapsto \langle \xi_i,f\xi_i\rangle
$$
will have the desired properties.  
\end{proof}

\begin{amendef}\label{amendef}
A space $X$ is \emph{amenable} if it satisfies the conditions in Lemma \ref{amenlem}.
\end{amendef}

It is not difficult to see that this is equivalent to the definitions of amenability in for example \cite[Section 3]{Block:1992qp} or \cite[Sections 3.3-3.6]{Roe:2003rw}

The equivalence of the first and third conditions in the proposition below is fairly well-known; the main point is that this equivalence still holds if one takes the `maximal spectrum'.

\begin{amen}\label{amen}
Let $E$ be a generating controlled set for $X$.  With notation as in Section \ref{lapsec}, the following are equivalent.
\begin{enumerate}
\item $0$ is in $\sigma_{l^2(X)}(\Delta^E)$;
\item $0$ is in $\sigma_{max}(\Delta^E)$;
\item $X$ is amenable.
\end{enumerate}
\end{amen}

\begin{proof}
It is clear that (1) implies (2).

To see that (2) implies (3), assume then that $0$ is an element of $\sigma_{max}(\Delta^E)$.  This is equivalent to $0$ being an element of the spectrum of $\Delta^E$ in the $C^*$-algebra $C^*_{u,max}(X)$ defined as the completion of $\C_u[X]$ for the norm
$$
\|T\|_{max}:=\sup\{\|\pi(T)\|_{\mathcal{B}(\mathcal{H})}~|~\pi:\C_u[X]\to\mathcal{B}(\mathcal{H})\text{ a representation}\}
$$
(the arguments of \cite[Section 3]{Gong:2008ja} show this supremum is finite for each $T\in \C_u[X]$).  Any point in the spectrum of a positive operator in a $C^*$-algebra can be realised as an eigenvalue in some representation, whence there exists a representation $\h$ of $C^*_{u,max}(X)$ (equivalently, of $\C_u[X]$) in which $0$ is an eigenvalue of $\Delta^E$.  Proposition \ref{conlem} then implies that this representation contains non-zero constant vectors.  Let $\xi$ be any norm-one constant vector, and let 
$$
\phi:\C_u[X]\to \C,~~~a\mapsto \langle \xi,a\xi\rangle
$$
be the corresponding vector state.  We will show that the restriction of $\phi$ to $l^\infty(X)$ is an invariant mean.  

Indeed, let $f$ be an element of $l^\infty(X)$, let $t:A\to B$ be a partial translation such that $B$ contains the support of $f$, and let $v$ be the operator corresponding to $t$.  Then $f\circ t=v^*fv$ and $vv^*f=fvv^*=f$ so the fact that $\xi$ is constant implies that
$$
\phi(f\circ t)=\phi(v^*fv)=\langle v\xi~,~fv\xi\rangle=\langle vv^*\xi~,~fvv^*\xi\rangle=\langle \xi~,~f\xi\rangle=\phi(f). 
$$

To see that (3) implies (1), let $(\xi_i)$ be a net of functions in $l^2(X)$ with the properties in part (2) of Lemma \ref{amenlem}.  Using Lemmas \ref{poslem} and \ref{laplem} we may write 
$$
\Delta^E=\sum_{j=1}^N v_jv_j^*+v_j^*v_j-v_j-v_j^*
$$
for some partial translations $v_1,...,v_N$.  For any $\xi_i$ we then have that 
$$
\langle \xi_i,\Delta^E\xi_i\rangle=\sum_{j=1}^N \langle \xi_i,(v_jv_j^*-v_j)\xi_i\rangle+\langle \xi_i,(v_j^*v_j-v_j^*)\xi_i\rangle,
$$
which tends to zero in the limit over $i$.  As $\Delta^E$ is a positive operator on $l^2(X)$, this implies that its spectrum contains zero.
\end{proof}

In particular, note that whether or not $0$ is in the above variations of the spectrum of $\Delta^E$ is a property not of $E$, but of the coarse space $X$.  

We now turn to the relationship between geometric property (T) and amenability.  First a lemma.

\begin{ameniso}\label{ameniso}
Let $X$ be a coarsely connected amenable space. Then for any generating controlled set $E$, $0$ is a non-isolated point of the spectrum of $\sigma_{max}(\Delta^E)$.  
\end{ameniso}

\begin{proof}
Using Proposition \ref{amen}, $0$ is also in $\sigma_{l^2(X)}(\Delta^E)$; it suffices to show $0$ is not an isolated point in $\sigma_{l^2(X)}(\Delta^E)$.  If it were, then $0$ would be an eigenvector of $\Delta^E$ for its action on $l^2(X)$.  Let $\xi$ be an eigenvector, and note that Proposition \ref{conlem} implies that $\xi$ is a fixed vector.  Let $x_0$ be a point in the support of $\xi$, and note that as $X$ is infinite, coarsely connected and monogenic, there exists a sequence
$$
x_0,x_1,x_2,...
$$
of distinct points in $X$ such that $(x_i,x_{i+1})$ is in $E$ for all $i$.  Let $v$ be the partial translation operator corresponding to the partial translation defined by
$$
t:\{x_n~|~n\geq 1\}\to \{x_n~|~n\geq 0\},~~~x_n\mapsto x_{n-1}.
$$
Then as $\xi$ is a fixed vector we have that
$$
(v\xi)(x_n)=(vv^*\xi)(x_{n-1})
$$
for all $n$, and thus by induction that all the values $\xi(x_n)$ are non-zero and equal.  This contradicts that $\xi$ is in $l^2(X)$.
\end{proof}

Finally, here is the characterisation of coarsely connected spaces with geometric property (T).  

\begin{infi}\label{infi}
Let $X$ be an infinite coarsely connected space.  Then $X$ has geometric property (T) if and only if it is not amenable.
\end{infi}

The special case of this result when $X$ is a group was proved in \cite[Lemma 7.2]{Willett:2010zh}.

\begin{proof}
If $X$ is amenable then by Lemma \ref{ameniso}, $0$ is a non-isolated point in $\sigma_{max}(\Delta^E)$ for any generating set $E$ for the coarse structure, and this contradicts geometric (T).  Conversely, if $X$ is not amenable, then $0$ is not in the spectrum of $\sigma_{max}(\Delta^E)$ for any generating set $E$ by Proposition \ref{amen}.  As the spectrum is closed and $\Delta^E$ is positive, it is contained in a set of the form $[c,\infty)$ for some $c>0$, and this in particular implies geometric property (T).
\end{proof}

This says that geometric property (T) is not very interesting for coarsely connected spaces!  In the next section, we will finally look at a class of interesting examples of spaces with geometric property (T).

\section{Relationship with property (T) groups}\label{exsec}

In this section we give some non-trivial examples of spaces with geometric property (T).  Up to trivial adjustments, these are the only examples we know.  Most of this material is contained in \cite[Section 7]{Willett:2010zh}, but fairly sketchily; we provide more detail here for the readers' convenience.  
 
It will be very convenient to use some $C^*$-algebraic machinery in this section, mainly as the following $C^*$-algebras are useful  to organize certain arguments.  This material was already briefly used in the proof of Proposition \ref{amen}.

\begin{spectra}\label{spectra}
Let $X$ be a space.  The \emph{uniform Roe algebra} if $X$, denoted $C^*_u(X)$, is the completion of $\C_u[X]$ for its natural $*$-representation on $l^2(X)$.  

The \emph{maximal uniform Roe algebra} of $X$, denoted $C^*_{u,max}(X)$, is the completion of $\C_u[X]$ for the norm
$$
\|T\|:=\sup\{\|\pi(T)\|_{\mathcal{B}(\mathcal{H})}~|~\pi:\C_u[X]\to \mathcal{B}(\mathcal{H}) \text{ a $*$-representation}\}
$$ 
(see \cite[Section 3]{Gong:2008ja} for a proof that this norm is finite).
\end{spectra}
Using the notation from Definition \ref{spectra0}, note that for any $T\in \C_u[X]$, $\sigma_{l^2(X)}(T)$ is the spectrum of $T$ considered as an element of $C^*_u(X)$, and $\sigma_{max}(T)$ is the spectrum of $T$ considered as an element of $C^*_{u,max}(X)$.

\begin{bsdef}\label{bsdef}
Let $\Gamma$ be an infinite finitely generated discrete group with a fixed finite generating set $S$.   Assume that $S=S^{-1}$.  
Let
$$
\Gamma=\Gamma_0\trianglerighteq \Gamma_1\trianglerighteq \Gamma_2\trianglerighteq\cdots 
$$
be a nested sequence of finite index normal subgroups of $\Gamma$ such that $\cap_n \Gamma_n=\{e\}$.  For each $n$, set $X_n=\Gamma/\Gamma_n$, and set 
$$
X=\sqcup_{n\in \N}X_n.
$$
Set 
$$
E_n=\{(x,y)\in X_n\times X_n~|~x^{-1}y\in S\}
$$
and
$$
E=\sqcup E_n\subseteq X\times X.
$$
Finally, equip $X$ with the monogenic coarse structure generated by $E$.  
\end{bsdef}

The following theorem characterises when a space built from a group as above has geometric property (T).

\begin{tgeot}\label{tgeot}
Let $X$ be a space built from data  $(\Gamma,(\Gamma_n))$ as above.  Then $X$ has geometric property (T) if and only if $\Gamma$ has property (T).
\end{tgeot}

In order to prove this, we need a lemma.  Let $\C[\Gamma]$ denote the complex group algebra of $\Gamma$, and note that the right actions of $\Gamma$ on the various $X_n$ give rise to a $*$-homomorphism
\begin{equation}\label{groupin}
\iota: \C[\Gamma]\to \C_u[X].
\end{equation}
This $*$-homomorphism is injective as $\cap_n\Gamma_n=\{e\}$.  Moreover, if $C^*_{max}(\Gamma)$ denotes the completion of $\C[\Gamma]$ for the norm
$$
\Big\|\sum_{g\in \Gamma}z_gg\Big\|:=\sup\Big\{\Big\|\Big(\sum_{g\in \Gamma}z_g\pi(g)\Big)\Big\|_\h~|~\pi:\Gamma\to \h \text{ a unitary representation}\Big\},
$$
then $\iota$ also induces a $*$-homomorphism $\iota:C^*_{max}(\Gamma)\to C^*_{u,max}(X)$ by the universal property of $C^*_{max}(\Gamma)$.  We have the following injectivity result, which is stronger than the statement that the map in line \eqref{tgeot} is injective.

\begin{maxgp}\label{maxgp}
The $*$-homomorphism $\iota:C^*_{max}(\Gamma)\to C^*_{u,max}(X)$ is injective.
\end{maxgp}

\begin{proof}
Note that the algebraic direct sum $\oplus_n \C_u[X_n]$ is an ideal in $\C_u[X]$; let $I$ denote its closure in $C^*_{u,max}(X)$.   It follows from the argument of \cite[Proposition 2.8]{Oyono-Oyono:2009ua} that 
$$
C^*_{u,max}(X)/I\cong \big( l^\infty(X)/C_0(X)\big)\rtimes_{max}\Gamma,
$$
where the right-hand-side denotes the maximal crossed product defined using the action of $\Gamma$ on $l^\infty(X)/C_0(X)$ induced by the right action on $X$.  It suffices to prove that the composed map
$$
C^*_{max}(\Gamma)\to C^*_{u,max}(X)\to \big(l^\infty(X)/C_0(X)\big)\rtimes_{max}\Gamma
$$
is an injection.

Now, for each $n$, let $\xi_n$ be the normalised characteristic function of $X_n$ in $l^2(X)$, and let 
$$
\phi_n:l^\infty(X)\to \C,~~~f\mapsto \langle \xi_n,f\xi_n\rangle
$$
be the corresponding vector state.  Let $\phi$ be any cluster point of the sequence $(\phi_n)$ of vector states on $l^\infty(X)$, and note that $\phi$ descends to a state on $l^\infty(X)/C_0(X)$.  It is $\Gamma$-invariant, as all the $\phi_n$ are.  Finally, consider the maps
$$
\C\stackrel{1}{\to} l^\infty(X)/C_0(X)\stackrel{\phi}{\to}\C,
$$
where the first map is the unit inclusion which is split by the ucp map $\phi$.  As maximal crossed products are functorial for ucp maps (see e.g.\ \cite[Exercise 4.1.4]{Brown:2008qy}), this gives rise to maps
$$
C^*_{max}(\Gamma)\to \big(l^\infty(X)/C_0(X)\big)\rtimes_{max}\Gamma\to C^*_{max}(\Gamma)
$$
whose composition is the identity; the first map is thus injective.
\end{proof}

\begin{injrem}\label{injrem}
The above proof is a disguised version of the following fact: we guess this is known, but do not know if it appears in the literature.  Let $G$ be a locally compact group acting on a compact Hausdorff topological space $X$.  Then the canonical $*$-homomorphism
$$
C^*_{max}(G)\to C(X)\rtimes_{max}G
$$
is injective if and only if there is an invariant measure on $X$.
\end{injrem}

\begin{proof}[Proof of Theorem \ref{tgeot}]
Consider the element 
$$
\Delta_\Gamma:=\sum_{s\in S}1-[s]\in \C[\Gamma]
$$
of the group algebra of $\Gamma$, and let $E$ be the controlled set appearing in the definition of a box space (Definition \ref{bsdef}).  Then the image of $\Delta_\Gamma$ under $\iota$ is the Laplacian $\Delta^E$ associated to $E$.  As $\iota$ is injective on the level of maximal completions (and injective maps of $C^*$-algebras preserve spectra) it follows that the spectrum of $\Delta_\Gamma$ in $C^*_{max}(\Gamma)$ is equal to  $\sigma_{max}(\Delta^E)$.  However, it is well known that $\Gamma$ has property (T) if and only if the spectrum of $\Delta_\Gamma$ in $C^*_{max}(\Gamma)$ is contained in a set of the form $\{0\}\sqcup [c,\infty)$ for some $c>0$ (see for example \cite[Theorem 3.2]{Valette:1984wy}); as $E$ is generating (and using Proposition \ref{nclem}), this is equivalent to geometric property (T) for $X$. 
\end{proof}

\begin{expgt}\label{expgt}
For spaces built from sequences of quotients as above, having geometric property (T) is a strictly stronger property that being an expander.
\end{expgt}

\begin{proof}
A space associated to the pair $(\Gamma,(\Gamma_n))$ is an expander if and only if the pair has property $(\tau)$: see for example \cite[Theorem 4.3.2]{Lubotzky:1994tw}.  The result now follows as there are many pairs $(\Gamma, (\Gamma_n))$, for example with $\Gamma$ a free group, which have property $(\tau)$ where $\Gamma$ does not have property (T).  
\end{proof}

Note also that whether a space as above has property (T) depends only on the ambient group $\Gamma$, not the given sequence of subgroups; on the other hand, whether or not such a space is an expander in general does depend on the choice of sequence of subgroups.

\section{Geometric property (T) and coarse\\ a-T-menability properties}\label{atmensec}
 
It follows from Lemma \ref{fin} and Corollary \ref{infi} that geometric property (T) is only really interesting when a space $X$ admits a decomposition
$$
X=\sqcup_{n\in\N}X_n
$$
into non-empty finite coarse components such that $|X_n|$ tends to infinity.  We assume for simplicity\footnote{Slightly more general results, for example allowing index sets other than $\N$, or only assuming that $|X_n|$ is unbounded, are certainly possible but we did not think that the extra messiness this would force on the statements is worth it.} throughout this section that we are dealing with a space of this form.  

Given such a space, we define its \emph{box space}  $\Box X$ to be the set $X$ equipped with the coarse structure generated by the original coarse structure on $X$, and all the singletons $\{(x,y)\}$ as $x$ and $y$ vary across $X$ (note that this coarse structure is never monogenic). 

Our goal in this section is to show that geometric property (T) is incompatible with the following notions of `coarse a-T-menability' for $X$: $X$ admits a coarse embedding into Hilbert space \cite{Yu:200ve}; $X$ admits a fibered coarse embedding into Hilbert space \cite{Chen:2012uq}; the restriction of the coarse groupoid of $X$ to its boundary is a-T-menable \cite{Finn-Sell:2012fk}.  

\begin{atthe}\label{atthe}
Assume that $X=\sqcup X_n$ splits into finite coarse components such that $|X_n|$ tends to infinity as above, and $X$ has geometric property (T).  Then the following are impossible:
\begin{enumerate}
\item $\Box X$ admits a coarse embedding into Hilbert space;
\item $\Box X$ admits a fibered coarse embedding into Hilbert space;
\item the restriction of the coarse groupoid of $X$ (or $\Box X$) to its boundary is a-T-menable.
\end{enumerate}
\end{atthe}

Natural examples satisfying conditions (2) and (3) are sequences of finite quotients of a-T-menable groups, and sequences of graphs $(X_n)$ such that the girth\footnote{i.e.\ length of shortest non-trivial cycle.} of $X_n$ tends to infinity: see \cite[Examples 2.4 and 2.5]{Chen:2012uq}.  It follows from the theorem that such spaces cannot have geometric property (T).  On the other hand, note that there are many expander sequences with girth tending to infinity; this gives another difference between geometric property (T) and general sequences of expanding graphs.

Special cases of this theorem follow from known results in $K$-theory \cite{Oyono-Oyono:2009ua,Willett:2010ud,Willett:2010zh,Chen:2012uq,Finn-Sell:2012fk,Finn-Sell:2013yq}, but the proof we give here is more direct and a little more general.  The theorem is definitely \emph{not} true for coarsely connected spaces: this follows from Corollary \ref{infi}.

The basic idea of the proof is to show that any of the coarse a-T-menability properties appearing in the statement allow one to construct $*$-representations of $\C_u[X]$ that contradict geometric property (T).  Unfortunately, properties (2) and (3) from the above theorem are quite technical, and are not stated anywhere in the literature in a form that is particularly well-suited for our purposes; as a result, in order to keep the proof of Theorem \ref{atthe} reasonably short and self-contained, we have had to be a little ad-hoc in some constructions below.  

In order to cover part (3) of the above theorem, we must use the language of the \emph{Stone-\v{C}ech compactification} of $X$; we thus start by recalling the relevant facts.  Let $Y$ be a discrete topological space.  The \emph{Stone-\v{C}ech compactification of $Y$}, denoted $\beta Y$ is a compact Hausdorff space containing $Y$ as a dense open subset.  It is determined by the following universal property: for any compact Hausdorff space $K$ and any function $f:Y\to K$, there is a unique continuous extension $f:\beta Y\to K$.  We write $\partial Y:=\beta Y\setminus Y$ for the associated \emph{Stone-\v{C}ech corona}.  Note that the universal property implies that for any $A\subseteq Y$, the inclusion map $A\to Y$ extends to an injection $\beta A\to \beta Y$; in particular, the closure of $A$ in $\beta Y$ is canonically identified with $\beta A$.

Let now $\mathcal{E}$ denote the coarse structure on $X\times X$.  For each controlled set $E\in\mathcal{E}$, let $\overline{E}$ denote its closure in $\beta X\times \beta X$.  Define also $\partial E$ to be the intersection $\overline{E}\cap (\partial X\times \partial X)$.  Define
$$
\beta_\mathcal{E}X:=\cup_{E\in \mathcal{E}}\overline{E}.
$$
We think of this space as a subset of $\beta X\times \beta X$ (and correspondingly write elements as pairs $(\omega_1,\omega_2)$), but equip it with the (`weak') topology determined by the following condition: a subset $U$ of $\beta_\mathcal{E}X$ is open if and only if its intersection with each $\overline{E}$ is open in $\beta X\times \beta X$.
Set 
$$
\pex:=\cup_{E\in \mathcal{E}}\partial E.
$$
equipped with the subspace topology from $\beta_\mathcal{E}X$.

With this topology, $\beta_\mathcal{E}X$ is a locally compact Hausdorff space, and $\pex$ is a closed subspace: $\beta_\mathcal{E}X$ actually identifies as a topological space with the \emph{coarse groupoid} $G(X)$, and $\pex$ with the `restriction to the boundary' of $G(X)$.  We will not use this, but see \cite{Skandalis:2002ng} or \cite[Chapter 10]{Roe:2003rw} for more information.

The following definition is easily seen to be equivalent to the property in part (3) of Theorem \ref{atthe}; we state it in this form to avoid having to introduce a lot of groupoid language.

\begin{bhp}\label{bhp}
The space $X$ is \emph{boundary a-T-menable} if there exists a continuous function $k:\partial_\mathcal{E}X\to\R_+$ such that the following hold.
\begin{itemize}
\item The function $k$ is \emph{normalized}: $k(\omega,\omega)=0$ for all pairs $(\omega,\omega)\in \partial_\mathcal{E}X$.
\item The function $k$ is \emph{symmetric}: $k(\omega_1,\omega_2)=k(\omega_2,\omega_1)$ for all $(\omega_1,\omega_2)\in \pex$.
\item The function $k$ is \emph{negative type}: for any finite subset $\{\omega_1,...,\omega_n\}$ of $\partial X$ such that all the pairs $(\omega_i,\omega_j)$ belong to $\partial_\mathcal{E} X$ and any finite subset $\{z_1,...,z_n\}$ of $\C$ such that $\sum_{i=1}^n z_i=0$ we have
$$
\sum_{i,j=1}^n \overline{z_i}z_j k(\omega_i,\omega_j)\leq 0.
$$
\item The function $k$ is \emph{proper}: if 
$$
c_E:=\inf\{k(\omega_1,\omega_2)~|~(\omega_1,\omega_2)\in \partial_\mathcal{E}X\setminus \partial E\}
$$ 
then the limit over the directed set of controlled sets (ordered by inclusion) $\lim_{E\in \mathcal{E}}c_E$ is infinity.
\end{itemize}
\end{bhp}

The main result we want to prove in this section is as follows.

\begin{girth}\label{girth}
Assume that $X=\sqcup X_n$ splits into finite coarse components such that $|X_n|$ tends to infinity, and $X$ is boundary a-T-menable.  Then $X$ does not have  geometric property (T).
\end{girth}

Before we prove this, we show how it implies Theorem \ref{atthe}


\begin{proof}[Proof of Theorem \ref{atthe}]
As already remarked, given the definition of the coarse groupoid $G(X)$ (see \cite{Skandalis:2002ng} or \cite[Chapter 10]{Roe:2003rw}), it is clear that $X$ is boundary a-T-menable in our sense if and only if the restriction of $G(X)$ to its boundary is a-T-menable in the sense of \cite[Section 3]{Tu:1999bq} (see also \cite[Section 5]{Skandalis:2002ng} and \cite{Finn-Sell:2012fk}).   The result of \cite[Corollary 20]{Finn-Sell:2013yq} thus implies that if $\Box X$ admits a \emph{fibered coarse embedding into Hilbert space} in the sense of \cite{Chen:2012uq} (and in particular if $\Box X$ admits a coarse embedding into Hilbert space in the sense of \cite{Yu:200ve}), then $X$ is boundary a-T-menable.  Theorem \ref{atthe} follows.
\end{proof}

The remainder of this section is devoted to the proof of Theorem \ref{girth}.  We assume from now on that $X$ is as in the statement of Theorem \ref{girth}, and assume that $k:\pex\to\C$ is as in the definition of boundary a-T-menability.  

For each $t>0$ define a function $k_t:\pex\to[0,1]$ by
\begin{equation}\label{kt}
k_t:(\omega_1,\omega_2)\mapsto e^{-tk(\omega_1,\omega_2)}.
\end{equation}

\begin{kposlem}\label{kposlem}
The functions $k_t$ from line \eqref{kt} have the following properties.
\begin{enumerate}
\item They are \emph{normalized}: $k_t(\omega,\omega)=1$ for all $\omega\in \partial X$.
\item They are \emph{symmetric}: $k_t(\omega_1,\omega_2)=k_t(\omega_2,\omega_1)$ for all $(\omega_1,\omega_2)\in \partial_\mathcal{E} X$.
\item They are \emph{positive type}: for any finite subset $\{\omega_1,...,\omega_n\}$ of $\partial X$ such that all the pairs $(\omega_i,\omega_j)$ are in $\pex$, and any finite subset $\{z_1,...,z_n\}$ of $\C$,
$$
\sum_{i,j=1}^n z_i\overline{z_j} k_t(\omega_i,\omega_j)\geq 0.
$$
\end{enumerate}
\end{kposlem}

\begin{proof}
Parts (1) and (2) are obvious.  Part (3) is essentially a version of a well-known theorem of Schoenberg (see for example \cite[Theorem C.3.2]{Bekka:2000kx}): it follows from the statement of \cite[Theorem C.3.2]{Bekka:2000kx} by applying that result separately to each finite subset $\{\omega_1,...,\omega_n\}$ of $\partial X$ such that $(\omega_i,\omega_j)\in \pex$ for all $i,j$.
\end{proof}

We will now use the functions $k_t$ to construct $*$-representations of $\C_u[X]$.  First, extend\footnote{The exact extensions we use will not affect the representations we build; however, we will make slightly refined choices of extension below in order to analyze properties of the representations.} the functions $k_t:\pex\to[0,1]$ to continuous functions $k_t:\beta_\mathcal{E}X\to [0,1]$; we may assume that $k_t(x,x)=1$ for all $x\in X$.  Define a form
$$
\langle\langle,\rangle\rangle_t :\C_u[X]\times\C_u[X]\to l^\infty(X)=C(\beta X)
$$
via the formula
$$
\langle\langle S,T\rangle \rangle_t(x)= \sum_{(y,z)\in X\times X}\overline{S_{xy}}T_{xz}k_t(y,z);
$$
as $X$ has bounded geometry the sum contains uniformly finitely many terms for each $x$, so this is well-defined.

In order to use these functions to build representations of $\C_u[X]$, we need some preliminaries.  First, we have the following lemma about elementary controlled sets.

\begin{cslem}\label{cslem}
Let $E$ be an elementary controlled set on $X$ which is the graph of a partial translation $t:A\to B$.  Let $\overline{t}:\overline{A}\to\overline{B}$ denote the extension of $t$ to the Stone-\v{C}ech compactifications. Then the closure $\overline{E}$ of $E$ in $\beta X\times \beta X$ is the set
$$
\{(t(\omega),\omega)~|~\omega\in \overline{A}\},
$$
which identifies homeomorphically with $\beta E$.
\end{cslem}

\begin{proof}
Denote by $g:A\to X\times X$ the `graph bijection' $g:x\mapsto (t(x),x)$ with image $E$.  Consider the maps
$$
A\stackrel{g}{\to}\beta X\times \beta X\stackrel{\pi}{\to} \beta X,
$$
where $\pi$ is the projection onto the second factor.  The composition of these maps is just the inclusion of $A$ into $\beta X$.  Now, the universal property of the Stone-\v{C}ech compactification gives rise to maps
$$
\overline{A}\stackrel{\overline{g}}{\to}\beta X\times \beta X\stackrel{\pi}{\to} \beta X,
$$
where the image of $\overline{g}$ is $\overline{E}$.  Uniqueness of the extension $\overline{g}$ implies that it must be equal to the map $\omega\mapsto (\overline{t}(\omega),\omega)$, which gives the characterisation of $\overline{E}$.  On the other hand, $\pi\circ \overline{g}$ is the identity inclusion $\beta A=\overline{A}\to\beta X$ (by uniqueness again), which implies that $\overline{g}$ is injective.  Hence $\overline{E}$ identifies canonically with $\beta A$, and so with $\beta E$.
\end{proof}

Now, let $f:X\times X\to\C$ be a bounded function with support in a controlled set $E$.  Using Lemma \ref{elde}, we may write $E=E_1\sqcup \cdots \sqcup E_n$, where each $E_i$ is elementary.  Say $E_i$ is the graph of the partial translation $t_i:A_i\to B_i$.  For each $i=1,...,n$, define $g_i:A_i\to \C$ by $g_i(x)=f_i(t_i(x),x)$, and extend $g_i$ to $\overline{A_i}$.  The extension of $f|_{E_i}$ to $\overline{E_i}\cong \beta E_i$ must be given by
$$
(t(\omega),\omega)\mapsto g(\omega)
$$
by uniqueness.  This formula then extends $f$ to all of $\overline{E}=\overline{E_1}\cup\cdots \cup \overline{E_n}$.  Using subdivisions, it is not difficult to see that this extension does not depend on the choice of decomposition $E=E_1\sqcup \cdots \sqcup E_n$.  

If $T$ is any operator in $\C_u[X]$ supported in a controlled set $E$, and $(\omega_1,\omega_2)$ is an element of $\partial E$, we define $T_{\omega_1\omega_2}$ using the extension process above applied to the function from $E$ to $\C$ defined by $(x,y)\mapsto T_{xy}$.  

For a controlled set $E$, we also define 
\begin{equation}\label{nofe}
N(E):=\max_{x\in X}|\{y\in X~|~(x,y)\in E\cup E^{-1}\}|
\end{equation}
and for an element $T$ of $\C_u[X]$, define 
\begin{equation}\label{noft}
N(T):=N(\{(x,y)\in X\times X~|~T_{xy}\neq 0\}).
\end{equation}

\begin{funprops}\label{funprops}
For each $t>0$ the form
$$
\langle\langle,\rangle\rangle_t :\C_u[X]\times\C_u[X]\to l^\infty(X)=C(\beta X)
$$
has the following properties.
\begin{enumerate}
\item The form $\langle\langle,\rangle\rangle_t$ is linear in the second variable and conjugate linear in the first.
\item For any $S,T\in \C_u[X]$ the $l^\infty$-norm of $\langle\langle S,T\rangle \rangle_t$ is bounded by
$$
\|\langle\langle S,T\rangle \rangle_t\|\leq \sup_{x,y\in X}|T_{xy}|\sup_{x,y\in X}|S_{xy}|N(T)N(S).
$$
\item The restriction of $\langle\langle S,T\rangle \rangle_t$ to $\partial X$ is given by the formula
$$
\langle\langle S,T\rangle \rangle_t(\omega)= \sum_{(\omega_1,\omega_2)\in \pex}\overline{S_{\omega\omega_1}}T_{\omega\omega_2}k_t(\omega_1,\omega_2).
$$
\item For any $S$ in $\C_u[X]$, the restriction of $\langle\langle S,S\rangle\rangle_t$ to $\partial X$ only takes non-negative values.
\end{enumerate}
\end{funprops}

\begin{proof}
Part (1) is clear.  Part (2) follows from the fact that $k_t$ takes values in $[0,1]$ and the triangle inequality.  

For part (3), note that any $S\in \C_u[X]$ is a finite sum of operators with the property that 
$$
\{(x,y)~|~S_{xy}\neq 0\}
$$
is elementary.  Using part (1), it suffices to assume that $S$ and $T$ have this property.  Assuming this, let $s:A_s\to B_s$ be the partial translation corresponding to $S$, and define $f:A_s\to\C$ by $f(x)=S_{s(x)x}$.  Similarly, define $g:A_t\to\C$ by $g(x)=T_{t(x)x}$, where $t:A_t\to B_t$ is the partial translation corresponding to $T$.  Then for any fixed $\omega\in \partial X$ we have via the discussion preceding this lemma that
$$
S_{\omega\omega'}=\left\{\begin{array}{ll} f(s^{-1}(\omega)) & \omega\in\overline{B_s} \text{ and } \omega=s(\omega') \\ 0 & \text{otherwise}\end{array}\right.
$$ 
and similarly for $T$. We thus have
\begin{align*}
\sum_{(\omega_1,\omega_2)\in \pex} & \overline{S_{\omega\omega_1}}T_{\omega\omega_2}k_t(\omega_1,\omega_2) \\ &=\left\{\begin{array}{ll}\overline{f(s^{-1}(\omega))}g(t^{-1}(\omega))k_t(s^{-1}(\omega),t^{-1}(\omega)) & \omega\in \overline{B_t}\cap \overline{B_s} \\ 0 & \text{otherwise}\end{array}\right.
\end{align*}
On the other hand, for any $x\in X$,
\begin{align*}
\langle\langle S,T\rangle \rangle_t(x) & =\sum_{(y,z)\in X\times X}\overline{S_{xy}}T_{xz}k_t(y,z) \\ & =\left\{\begin{array}{ll} \overline{f(s^{-1}(x))}g(t^{-1}(x))k_t(s^{-1}(x),t^{-1}(x)) & x\in B_t\cap B_s \\ 0 & \text{otherwise}\end{array}\right.
\end{align*}
The claimed formula follows.

Part (4) follows from part (3) and Lemma \ref{kposlem}. 
\end{proof}

Now, for each $n$, define a state on $l^\infty(X)$ by the formula
$$
\phi_n:l^\infty(X)\to\C,~~~f\mapsto \frac{1}{|X_n|}\sum_{x\in X_n}f(x).
$$
Let $\phi$ be any cluster point of the sequence $(\phi_n)$ in the state space of $l^\infty(X)$.  Note that $\phi$ descends to a state on the quotient $l^\infty(X)/C_0(X)$, which naturally identifies with $C(\partial X)$.

For each $t>0$ we may thus define a form on $\C_u[X]$ by
$$
\langle S,T\rangle_t=\phi(\langle\langle S,T\rangle \rangle_t).
$$
Using Lemma \ref{funprops} parts (1) and (2), and that $\phi$ is a state, each form $\langle,\rangle_t$ is linear in the second variable, conjugate linear in the first, and positive semi-definite; separation and completion thus defines a Hilbert space $\mathcal{H}_t$ for each $t>0$.  An element $S\in \C_u[X]$ gives rise to an equivalence class $[S]$ in this Hilbert space.  Provisionally define a representation $\pi_t$ of $\C_u[X]$ on $\mathcal{H}_t$ via the formula
$$
\pi_t(T):[S]\mapsto [TS].
$$
Lemma \ref{funprops} part (3), and that $\phi$ has norm one, implies that each operator $\pi_t(T)$ extends to a bounded linear operator on $\mathcal{H}_t$, and thus we have a well-defined map
\begin{equation}\label{pit}
\pi_t:\C_u[X]\to\mathcal{B}(\mathcal{H}_t).
\end{equation}

\begin{pitlem}\label{pitlem}
The map $\pi_t$ in line \eqref{pit} above is a $*$-representation.  It does not depend on the choice of extension of $k_t$.
\end{pitlem}

\begin{proof}
Linearity and multiplicativity of $\pi_t$ are clear, so to show that $\pi_t$ is a $*$-representation it suffices to check that it preserves adjoints.  

As $\phi$ is cluster point of the functionals $\phi_n$, it suffices to show that 
$$
\phi_n(\langle\langle  R^*S,T \rangle\rangle_t)=\phi_n(\langle\langle  S,RT \rangle\rangle_t)
$$
for all $n$ and all $R,S,T\in \C_u[X]$.  Computing,  
\begin{align*}
\phi_n(\langle\langle  R^*S,T \rangle\rangle_t)&=\frac{1}{|X_n|}\sum_{x\in X_n}\sum_{y,z\in X_n}\sum_{u\in X_n}\overline{R^*_{xu}S_{uy}}T_{xz}k_t(y,z) \\
&=\frac{1}{|X_n|}\sum_{x,y,z,u\in X_n}\overline{S_{uy}}R_{ux}T_{xz}k_t(y,z) \\
&=\frac{1}{|X_n|}\sum_{x\in X_n}\sum_{y,z\in X_n}\overline{S_{xy}}\sum_{u\in X_n}R_{xu}T_{uz}k_t(y,z) \\
&=\phi_n(\langle\langle  S,RT \rangle\rangle_t).
\end{align*}
The fact that $\pi_t$ does not depend on the choice of extension of $k_t$ follows from Lemma \ref{funprops} part (3) and that $\phi$ only depends on the restriction of a function in $C(\beta X)$ to $\partial X$.
\end{proof}

Our eventual goal is to preclude geometric property (T) by showing that the representations $\mathcal{H}_t$ `come close' to containing constant vectors for small $t>0$, although none of them actually do contain constant vectors.  The following lemma is the next step.

\begin{invlem}\label{invlem}
For any $t>0$, the $*$-representation $\pi_t:\C_u[X]\to\mathcal{B}(\mathcal{H}_t)$ contains no constant vectors.
\end{invlem}

In order to prove this, we need a combinatorial lemma.

\begin{graphlem}\label{graphlem}
Let $E$ be a symmetric generating set for the coarse structure on $X$ that contains the diagonal, and fix $r\in \N$.   Then there exists $s,N\in \N$ such that for all $n\geq N$ there exists a bijective partial translation $t_n:X_n\to X_n$ such that 
$$
\text{graph}(t_n)\subseteq E^{\circ s}\setminus E^{\circ r}.
$$
\end{graphlem}

\begin{proof}
Let $r$ be given, and let $s$ be so large that $\lfloor s/3\rfloor-r\geq N(E^{\circ r})$, where $N(E^{\circ r})$ is as in line \eqref{nofe}.  As $|X_n|$ tends to infinity, and using the bounded geometry assumption there exists $N$ such that for all $n\geq N$ and all points $x\in X_n$, there is a point $y=y(x)\in X_n$ such that $(x,y)\not\in E^{\circ \lfloor s/3\rfloor }$.  Fix $n\geq N$, and let $G$ be the graph with vertex set $X_n$ and where two vertices $x,y$ are connected by an edge if and only if $(x,y)$ is in $E^{\circ s}\setminus E^{\circ r}$.  It suffices to show that there is a bijection $\sigma:X_n\to X_n$ such that $(x,\sigma(x))$ is an edge in $G$ for all $x\in X_n$.  It suffices by Tutte's 2-matching theorem (see for example \cite[Proposition 2]{Levine:2001fk}) to show that if $C$ is a subset of $X_n$ no two vertices of which are connected by an edge in $G$, and if we set 
$$
d(C)=\{x\in X_n~|~\text{ there exists } y\in X_n \text{ such that } (x,y) \text{ is an edge of $G$}\}
$$
then $|d(C)|\geq |C|$.

Fix then such a set $C$, and define a relation on $C$ by $x\sim y$ if and only if $(x,y)\in E^{\circ r}$.  This is an equivalence relation: it is symmetric and reflexive as $E$ is symmetric and contains the diagonal.  It is transitive as if $x\sim y$ and $y\sim z$ then $(x,z)$ is in $E^{\circ 2r}$; as $x,z$ are in $C$, they are not connected by an edge in $G$, and so this is impossible unless $(x,z)$ is actually in $E^{\circ r}$.  Fix a subset $C_0$ of $C$ containing one representative of each equivalence class, and for each $x\in C_0$, define
$$
C_x:=\{y\in C~|~(x,y)\in E^{\circ r}\}
$$  
to be its equivalence class, which has at most $N(E^{\circ r})$ members.  For each $x\in C_0$, define
$$
D_x:=\{y\in X_n~|~(x,y)\in E^{\circ\lfloor s/3\rfloor}\setminus E^{\circ r}\}.
$$
The choice of $N$ implies that there exists $y\in X_n$ such that $(x,y)\not\in E^{\circ \lfloor s/3\rfloor}$.  As $E$ is generating, it follows that there is $m\geq \lfloor s/3\rfloor$ and a sequence of distinct points
$$
x=x_0,x_1,...,x_m=y
$$
such that $(x_i,x_{i+1})$ is in $E$ for $i=0,...,m-1$, and $(x,x_i)\not\in E^{\circ i-1}$ for $i=1,...,m$.  Hence in particular $x_{r+1},...,x_{\lfloor s/3\rfloor}$ are in $D_x$, and thus 
$$
|D_x|\geq \lfloor s/3\rfloor-r \geq N(E^{\circ r})\geq |C_x|,
$$
where the central inequality follows by choice of $s$.  Finally, note that if $x$ and $y$ are distinct points in $C_0$, then $(x,y)$ is not an edge and $(x,y)\not\in E^{\circ r}$.  It follows by definition of $G$ that $(x,y)\not\in E^{\circ s}$ and in particular $D_x\cap D_y=\varnothing$.  Moreover, each $D_x$ is contained in $d(C)$ whence
$$
|d(C)|\geq \Big|\bigcup_{x\in C_0}D_x\Big|=\sum_{x\in C_0}|D_x|\geq \sum_{x\in C_0}|C_x|=|C|
$$
completing the proof.
\end{proof}

\begin{proof}[Proof of Lemma \ref{invlem}]
Fix $t>0$, and assume for contradiction that $\xi\in \mathcal{H}_t$ is a constant vector of norm one.  Let $[T]\in \mathcal{H}_t$ be a norm one element coming from $T\in \C_u[X]$ such that $\|[T]-\xi\|_{\mathcal{H}_t}<1/4$.  Let $E$ be a symmetric generating set for the coarse structure that contains the diagonal.  Let $K\in \N$ be such that $T_{xy}=0$ whenever $(x,y)\not\in E^{\circ K}$ and define $\tau:=\sup_{x,y\in X}|T_{xy}|$.

Let $r>2K$ be so large that whenever $(\omega_1,\omega_2)\not\in \partial E^{\circ (r-2K)}$, we have that  
$$
k_t(\omega_1,\omega_2)\leq \frac{1}{4N(T)^2\tau^2}
$$ 
where $N(T)$ is as in line \eqref{noft} (such an $r$ exists by properness of $k$).  Adjusting the extensions of $k_t$ to $\beta_\mathcal{E}X$ if necessary, we may assume that this estimate holds in the stronger form: 
\begin{equation}\label{ksmall}
\text{ for all }(y,z)\not\in E^{\circ (r-2K)},~~ k_t(y,z)\leq \frac{1}{4N(T)^2\tau^2}.
\end{equation}

Let $N$ and $s$ be as in Lemma \ref{graphlem} for $r$ as above, and let $t_n:X_n\to X_n$ be the bijective partial translation given by that lemma for $n\geq N$, and the empty partial translation otherwise.  Let $v\in \C_u[X]$ be the partial translation that is defined using $t_n$ on each $X_n$.  As all but finitely many of the $t_n$ are bijections, $\pi_t(v)$ is a unitary operator on $\mathcal{H}_t$.

Now, for any $n\geq N$ and $x\in X_n$,
\begin{align*}
\langle\langle T,vT\rangle\rangle_t(x)&=\sum_{(y,z)\in X\times X}\overline{T_{xy}}\sum_{u\in X}v_{xu}T_{uz}k_t(y,z) \\&=\sum_{(y,z)\in X\times X}\overline{T_{xy}}T_{t_n^{-1}(x)z}k_t(y,z)
\end{align*}
If the term $\overline{T_{xy}}T_{t_n^{-1}(x)z}$ is non-zero, then $(x,y)$ and $(t_n^{-1}(x),z)$ are in $E^{\circ K}$.  As $(x,t_n^{-1}(x))$ is not in $E^{\circ r}$, however, this forces $(y,z)\not\in E^{\circ (r-2K)}$, and thus by line \eqref{ksmall}, $k_t(x,y)\leq \frac{1}{4N(T)^2\tau^2}$.  Hence for all $x\in X_n$
$$
|\langle\langle T,vT\rangle\rangle_t(x)|\leq \sum_{(y,z)\in X\times X}|\overline{T_{xy}}||T_{t_n^{-1}(x)z}|k_t(y,z)\leq \frac{1}{4}
$$
for all $x\in X$, whence $\phi_n(\langle\langle T,vT\rangle\rangle_t)\leq 1/4$ for all $n$ and so
\begin{equation}\label{small}
|\langle [T],\pi_t(v)[T]\rangle_t|\leq 1/4.
\end{equation}  

On the other hand, the facts that $\pi_t(v)$ is unitary, $\xi$ is constant, $[T]$ and $\xi$ have norm one, and $\|[T]-\xi\|< \frac{1}{4}$ together imply that 
\begin{align*}
|\langle [T],\pi_t(v)[T]\rangle_t| & \geq|\langle \xi,\pi_t(v)\xi\rangle_t|-|\langle [T],\pi_t(v)(\xi-[T])\rangle_t|- |\langle [T]-\xi,\pi_t(v)\xi\rangle_t|\\
& > |\langle \xi,\pi_t(vv^*)\xi\rangle_t|-1/4 -1/4\\
& =1/2.
\end{align*}
This contradicts line \eqref{small}, so we are done.
\end{proof}

The following lemma completes the proof of Theorem \ref{girth}.

\begin{speclem}\label{speclem}
Let $\Delta^E$ be a Laplacian operator in $\C_u[X]$, and let $\epsilon>0$.  Then for all suitably small $t>0$, the spectrum of $\pi_t(\Delta^E)$ contains points from $[0,\epsilon]$.
\end{speclem}

\begin{proof}
Let $I$ be the identity operator in $\C_u[X]$.  Then for any $x\in X$, we have the formula
$$
\langle\langle I,\Delta_EI\rangle\rangle_t(x)=\sum_{(y,z)\in X\times X}\overline{I_{xy}}\Delta_{xz}k_t(y,z) =\sum_{z}\Delta_{xz}k_t(x,z),
$$ 
whence
\begin{equation}\label{taun}
\phi_n(\langle\langle I,\Delta_EI\rangle\rangle_t)=\frac{1}{|X_n|}\sum_{x,y\in X_n}\Delta_{xy}k_t(x,y).
\end{equation}
Now, by boundedness of $k$ on $\partial E$, there exists $t>0$ such that
$$
\sup_{(\omega_1,\omega_2)\in \partial E}|1-k_t(\omega_1,\omega_2)|<\frac{\epsilon}{N(E)^2}.
$$
We may assume without loss of generality that the extension of $k_t$ to $X\times X$ satisfies 
$$
\sup_{(x,y)\in \partial E}|1-k_t(x,y)|<\frac{\epsilon}{N(E)^2}. 
$$
Looking back at line \eqref{taun}, we have
$$
|\phi_n(\langle\langle I,\Delta_EI\rangle\rangle_t)| \leq \frac{1}{|X_n|}\Big|\sum_{x,y\in X_n}\Delta_{xy}\Big|+\frac{1}{|X_n|}\Big|\sum_{x,y\in X_n}\Delta_{xy}(1-k_t(x,y))\Big|.
$$
A simple computation shows that the first term on the right hand side is zero, however, whence 
$$
|\phi_n(\langle\langle I,\Delta_EI\rangle\rangle_t)|\leq \frac{1}{|X_n|}N(E)|X_n|\sup_{x,y\in E\cap (X_n\times X_n)}|\Delta_{xy}(1-k_t(x,y))|<\epsilon.
$$

Finally, as $\phi$ is a cluster point of the sequence $(\phi_n)$, it follows from this and that $\pi_t(\Delta^E)$ is a positive operator that $\langle [I],\pi_t(\Delta^E)[I]\rangle_t$ is in $[0,\epsilon]$; as the set of values
$$
\{\langle \xi,\pi_t(\Delta^E)\xi\rangle_t~|~\xi\in \mathcal{H}_t\}
$$
is contained in the convex hull of the spectrum of $\pi_t(\Delta^E)$, this completes the proof.
\end{proof}

\section{Questions and comments}\label{qsec}

We conclude the paper with some open questions and comments.

\begin{questions}
\begin{enumerate}
\item Does a `generic' sequence of graphs have geometric property (T)?  This is a strengthening of the well-known fact that a generic sequence of graphs is an expander (see e.g.\ \cite[Proposition 1.2.1]{Lubotzky:1994tw}).  It is also possibly connected to the fact that a `generic' hyperbolic group has property (T) \cite{Zuk:2003jf}.
\item Are there useful necessary and / or  sufficient conditions for geometric property (T) that can be stated purely in terms of graph theoretic properties?  To answer question 1, it is probably necessary to answer this question first.
\item Similarly, are there useful necessary and / or sufficient conditions for conditions (2) and / or (3) from Theorem \ref{girth} that can be stated purely in terms of graph theoretic properties, other than the known condition using girth?
\item We suspect that the results of Section \ref{atmensec} are really part of a general result about groupoids.  Precisely, treating $\C_u[X]$ as the convolution $*$-algebra $C_c(G(X))$ of the coarse groupoid $G(X)$ it is not too difficult to extrapolate the ideas of this paper to define a `topological property (T)' for this groupoid, and indeed any `reasonable' locally compact groupoid.  We then suspect that the results of Section \ref{atmensec} say that this `topological property (T)' is incompatible with a-T-menability in the presence of an invariant measure on the unit space of the groupoid (the existence of an invariant measure corresponds to the amenability of the space in the assumptions of Theorem \ref{bhp}).  It might be interesting to develop this further: for example, Theorem \ref{cinv} naturally corresponds to a statement about Morita invariance of the general `topological property (T)'; but we have no idea if the corresponding general result would be true.
\item Our version of geometric property (T) only really has good properties for disjoint unions of finite metric spaces; in the language of point (4) above, the issue is the presence of an invariant measure on the unit space of the groupoid.  Is there a property that has more interesting consequences in the context of more general metric spaces, e.g.\ including Cayley graphs of infinite groups? Compare for example \cite[Section 11.4.3]{Roe:2003rw}.
\end{enumerate}
\end{questions}

\bibliographystyle{abbrv}

\bibliography{Generalbib}

\end{document}